\documentclass[12pt, reqno]{amsart}

\usepackage{cancel}

\usepackage{mathtools}
\mathtoolsset{showonlyrefs}
\numberwithin{equation}{section}

\usepackage[dvipdfmx]{hyperref, graphicx}

\usepackage{mathtools}
\mathtoolsset{showonlyrefs}

\usepackage{amssymb}
\usepackage{amsthm}
\usepackage{amscd}
\usepackage{amsmath}
\usepackage{epic,eepic}
\usepackage{mathrsfs}
\usepackage{latexsym}
\usepackage{pifont}
\usepackage[all]{xy}
\usepackage {cancel}

\theoremstyle{plain}
\newtheorem{lemma}{Lemma}[section]
\newtheorem{prop}[lemma]{Proposition}
\newtheorem{thm}[lemma]{Theorem}
\newtheorem{cor}[lemma]{Corollary}
\newtheorem{intthm}{Theorem}

\newtheorem{prdef}[lemma]{Proposition-Definition}
\newtheorem{tdef}[lemma]{Theorem-Definition}

\theoremstyle{definition}

\newtheorem{rem}[lemma]{Remark}
\newtheorem{defi}[lemma]{Definition}
\newtheorem{exa}[lemma]{Example}

\newtheorem{problem}{Problem}

\newcommand{\bde}{\begin{defi}}
\newcommand{\ede}{\end{defi}\vspace{1mm}}
\newcommand{\ble}{\begin{lemma}}
\newcommand{\ele}{\end{lemma}}
\newcommand{\bpr}{\begin{prop}}
\newcommand{\epr}{\end{prop}}
\newcommand{\bt}{\begin{thm}}
\newcommand{\et}{\end{thm}}
\newcommand{\bco}{\begin{cor}}
\newcommand{\eco}{\end{cor}}
\newcommand{\bre}{\begin{rem}}
\newcommand{\ere}{\end{rem}}
\newcommand{\bex}{\begin{exa}}
\newcommand{\eex}{\end{exa}}
\newcommand{\bpf}{\begin{proof}}
\newcommand{\epf}{\end{proof}}

\newcommand{\mcA}{\mathcal{A}}
\newcommand{\mcB}{\mathcal{B}}
\newcommand{\mcC}{\mathcal{C}}
\newcommand{\mcD}{\mathcal{D}}
\newcommand{\mcE}{\mathcal{E}}
\newcommand{\mcF}{\mathcal{F}}
\newcommand{\mcG}{\mathcal{G}}
\newcommand{\mcH}{\mathcal{H}}

\newcommand{\mcK}{\mathcal{K}}
\newcommand{\mcL}{\mathcal{L}}
\newcommand{\mcM}{\mathcal{M}}
\newcommand{\mcN}{\mathcal{N}}
\newcommand{\mcO}{\mathcal{O}}

\newcommand{\mcS}{\mathcal{S}}
\newcommand{\mcT}{\mathcal{T}}
\newcommand{\mcU}{\mathcal{U}}

\newcommand{\mbC}{\mathbb{C}}

\newcommand{\mbF}{\mathbb{F}}
\newcommand{\mbG}{\mathbb{G}}
\newcommand{\mbH}{\mathbb{H}}

\newcommand{\mbP}{\mathbb{P}}
\newcommand{\mbQ}{\mathbb{Q}}

\newcommand{\mbS}{\mathbb{S}}

\newcommand{\mbV}{\mathbb{V}}

\newcommand{\mbY}{\mathbb{Y}}
\newcommand{\mbZ}{\mathbb{Z}}

\newcommand{\mfS}{\mathfrak{S}}

\newcommand{\mfc}{\mathfrak{c}}

\newcommand{\mfg}{\mathfrak{g}}

\newcommand{\mfl}{\mathfrak{l}}

\newcommand{\mfo}{\mathfrak{o}}
\newcommand{\mfp}{\mathfrak{p}}

\newcommand{\mfs}{\mathfrak{s}}

\newcommand{\msC}{\mathscr{C}}

\newcommand{\msE}{\mathscr{E}}
\newcommand{\msF}{\mathscr{F}}

\newcommand{\msL}{\mathscr{L}}

\newcommand{\msP}{\mathscr{P}}

\newcommand{\msX}{\mathscr{X}}

\newcommand{\SSP}{\vspace{3mm}}
\newcommand{\LSP}{\vspace{5mm}}
\newcommand{\mr}{\mathrm}

\newcommand{\Fus}{\rotatebox[origin=c]{180}{$\mbY$}}

\begin{document}

\title[The generic \'{e}taleness of the moduli space of dormant $\mathfrak{so}_{2\ell}$-opers]{The generic \'{e}taleness of \\ the moduli space  of dormant $\mathfrak{so}_{2\ell}$-opers}
\author{Yasuhiro Wakabayashi}
\markboth{}{}
\maketitle
\footnotetext{Y. Wakabayashi: 
Graduate School of Information Science and Technology, Osaka University, Suita, Osaka 565-0871, Japan;}
\footnotetext{e-mail: {\tt wakabayashi@ist.osaka-u.ac.jp};}
\footnotetext{2020 {\it Mathematical Subject Classification}: Primary 14H10, Secondary 14H60;}
\footnotetext{Key words: oper, moduli space, p-curvature, curve, positive characteristic, connection}
\begin{abstract}
The generic \'{e}taleness is an important property on the moduli space of dormant $\mathfrak{g}$-opers (for a simple Lie algebra $\mathfrak{g}$) in the context of enumerative geometry. In the previous study, this property has been verified under the assumption that $\mathfrak{g}$ is either $\mathfrak{sl}_\ell$, $\mathfrak{so}_{2\ell -1}$, or $\mathfrak{sp}_{2\ell}$ for any sufficiently small positive integer $\ell$. The purpose of the present paper is to prove the generic \'{e}taleness for one of the remaining cases, i.e., $\mathfrak{g} = \mathfrak{so}_{2\ell}$. As an application of this result, we obtain a factorization formula for computing the generic degree induced from pull-back along various clutching morphisms between moduli spaces of pointed stable curves.

\end{abstract}
\tableofcontents

\section{Introduction} \label{S1}
\LSP

Linear differential equations, or more generally flat connections,   {\it  in characteristic $p> 0$} (where $p$ is a prime number) have been greatly investigated  from a variety of perspectives.
For example, 
the study of such objects involving  the Grothendieck-Katz conjecture 
leads us 
 to 
understand complex  linear differential equations and the algebraicity of  their solutions 
    (cf. ~\cite{Kat2}, ~\cite{And}).
Also, characteristic-$p$
  versions of non-Abelian Hodge theory and the geometric Langlands correspondence  have been established by applying many  techniques  specific to flat connections  in positive characteristic (cf.  ~\cite{BrBe}, ~\cite{ChZh1}, ~\cite{ChZh2}, ~\cite{GLQ}, ~\cite{LSZ}, ~\cite{OgVo}, and ~\cite{She}).

A key ingredient in these works is
  the notion of {\it $p$-curvature}.
By definition, the $p$-curvature of a flat connection in characteristic $p$  measures the obstruction to the compatibility of $p$-power structures
appearing in certain associated spaces of infinitesimal symmetries.
In particular, some classes of flat connections  characterized by $p$-curvature 
 deserve   special attention, and it is natural to ask how many of them there are.

Our study primary 
concerns
 {\it dormant $\mfg$-opers} (for a simple  Lie algebra $\mfg$), which may be thought of as generalizations of  linear  homogenous ODEs with (unit principal symbol and)
   vanishing $p$-curvature (cf. ~\cite[Definitions 2.1 and  3.15]{Wak8} for the definition of a dormant $\mfg$-oper).
 Here, let us fix  a pair of nonnegative integers $(g, r)$ with $2g-2 +r > 0$,
and denote by $\overline{\mcM}_{g, r}$ the moduli stack of $r$-pointed stable curves of genus $g$ in characteristic $p$.
Then, we obtain the moduli stack
\begin{align} \label{eeRt2}
\mcO p_{\mfg, g, r}^{^\mr{Zzz...}}
\end{align}
(cf. ~\cite[Eq.\,(433)]{Wak8})
 classifying  pairs $(\msX, \msE^\spadesuit)$ consisting of a pointed stable curve   $\msX$ in $\overline{\mcM}_{g, r}$  and  a dormant $\mfg$-oper $\msE^\spadesuit$ on it.
The assignment 
$(\msX, \msE^\spadesuit)\mapsto \msX$ determines a projection
\begin{align} \label{eeRt1}
\Pi_{\mfg, g, r} :
\mcO p^{^\mr{Zzz...}}_{\mfg, g, r}\rightarrow \overline{\mcM}_{g, r},
\end{align}
 by which $\mcO p^{^\mr{Zzz...}}_{\mfg, g, r}$ may be considered as a stack over $\overline{\mcM}_{g, r}$.

This
stack 
  for $\mfg = \mfs \mfl_2$ was originally introduced and  investigated in the context of $p$-adic Teichm\"{u}ller theory (cf. ~\cite{Moc2}), in which  
 dormant $\mfs \mfl_2$-opers (or more generally, certain $\mfs \mfl_2$-opers with nilpotent $p$-curvature)
play an analogous  role to  ``nice" projective structures on Riemann surfaces such as those  arising   from uniformization.

 One central theme of our study is to find out how many dormant $\mfg$-opers there are on a fixed curve. 
It leads us to  investigate  the structure of $\mcO p^{^\mr{Zzz...}}_{\mfg, g, r}$, as well as of $\Pi_{\mfg, g, r}$. 
 For a general $\mfg$,
 it has been shown that $\mcO p^{^\mr{Zzz...}}_{\mfg, g, r}$
 is a  nonempty proper Deligne-Mumford  stack and 
 $\Pi_{\mfg, g, r}$ is finite
 (cf. ~\cite[Theorem C]{Wak8}).
(Under some restricted situations, the finiteness  was previously proved in ~\cite[Chap.\,II, Theorem 2.8]{Moc2} and  ~\cite[Corollary 6.1.6]{JP}.)
Moreover, we know the {\it generic \'{e}taleness} of $\Pi_{\mfg, g, r}$ when $\mfg$ is, e.g.,  $\mfg = \mfs \mfl_n$ with $2n < p$  (cf. ~\cite[Theorem G]{Wak8}).
In that case,
  it makes sense to speak of the generic degree $\mr{deg}(\Pi_{\mfg, g, r})$ of $\Pi_{\mfg, g, r}$,
which counts the number of dormant $\mfg$-opers on 
$\msX$ classified by a general geometric point of $\overline{\mcM}_{g, r}$.

Note that
the values $\mr{deg} (\Pi_{\mfg, g, r})$ for
 $\mfg = \mfs \mfl_2$ 
have been    explicitly  computed  as a consequence of 
establishing the remarkable correspondences
 between 
 the following three topics (with dormant $\mfs \mfl_2$-opers at the center):

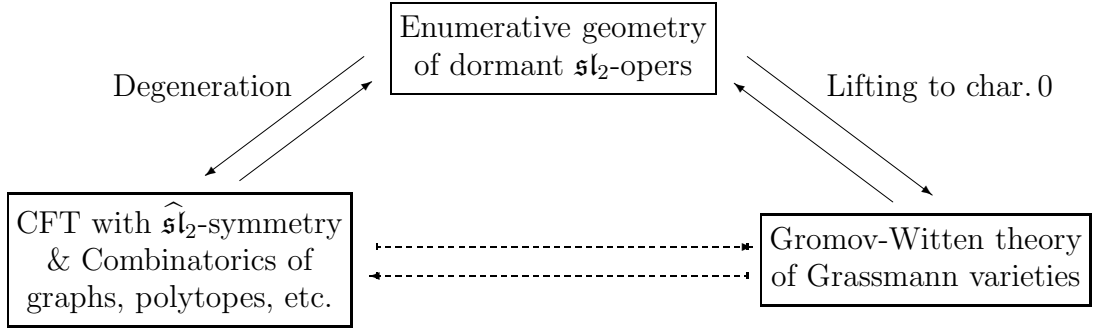
\vspace{5mm}
\begin{center}
\begin{picture}(365,120)

\put(125,90){\fbox{$\begin{matrix}
\text{Enumerative geometry} \\
\text{of dormant $\mfs \mfl_2$-opers} 
\end{matrix}$}}

\put(-20,10){\fbox{$\begin{matrix}
\text{CFT with $\widehat{\mfs \mfl}_2$-symmetry} \\
\text{\& Combinatorics of} \\
\text{graphs, polytopes, etc.} 
\end{matrix}$}}

\put(265,10){\fbox{$\begin{matrix}
\text{Gromov-Witten theory} \\
\text{of Grassmann varieties}
\end{matrix}$}}


\put(120,7){\vector(-1,0){3}}
\put(259,18){\vector(1,0){3}}

\put(120, 7){\dashbox{2.0}(140, 0){}}

\put(120, 18){\dashbox{2.0}(140, 0){}}


\put(68,43){\vector(4,3){50}}
\put(115,90){\vector(-4,-3){60}}

\put(260,90){\vector(4,-3){70}}
\put(315,35){\vector(-4,3){60}}

\put(20,75){$\begin{matrix}
\text{Degeneration}
\end{matrix}$}

\put(290,75){$\begin{matrix}
\text{Lifting to char.\,$0$}
\end{matrix}$}


\end{picture}
\vspace{5mm}
\end{center}
More specifically, 
by observing the behavior of dormant opers (including the case of $\mfg = \mfs \mfl_n$) when
 the underlying curve deforms or degenerates,
 one can carry out  the following (mutually independent) discussions (a)-(c), partially based on  
  methods and perspectives in $p$-adic Teichm\"{u}ller theory:
  \begin{itemize}
 \item[(a)]
 When the underlying curve totally degenerates (in the sense of \S\,\ref{SS3rr2} or ~\cite[Definition 7.15]{Wak8}),
 dormant $\mfs \mfl_2$-opers on that curve can be described 
 by using certain combinatorial objects, i.e.,   {\it balanced $p$-edge numberings} on a trivalent graph,  in the terminology of    ~\cite[Definition 3.1]{Wak31}.
 Moreover, 
 according to a work by F. Liu and B. Osserman (cf. ~\cite{LiOs}, ~\cite{Wak3}),
such numberings correspond to  lattice points inside a rational polytope.
 It follows that
 the numbers of these objects can be expressed as a  polynomial  with respect to ``$p$" by the classical Ehrhart theory, and 
 coincides with $\mr{deg}(\Pi_{\mfs \mfl_2, g, 0})$ because 
  $\Pi_{\mfs \mfl_2, g, 0}$ is \'{e}tale at the points classifying totally degenerate curves.
  That is to say,  there exists a degree $3g-3$ polynomial $H (t)$ in $\mbQ [t]$ (independent of $p$) such that 
 \begin{align}
\mr{deg}(\Pi_{\mfs\mfl_2, g, 0}) = \sharp  \left\{\begin{matrix} \text{balanced $p$-edge numberings on} \\ \text{a trivalent graph of type  $(g, 0)$} \end{matrix} \right\} = H (p).
\end{align}

\item[(d)]
To identify this value more explicitly, 
we also use the generic \'{e}taleness of $\Pi_{\mfs \mfl_2, g, r}$;  it gives a detailed understanding about 
a factorization property of $\Pi_{\mfs \mfl_2, g, r}$'s with respect to  degeneration of the underlying curve (i.e., pull-back along various clutching morphisms between  moduli spaces of pointed stable curves).
In particular,  we can compare  the resulting decompositions of  $\mr{deg}(\Pi_{\mfs \mfl_2, g, r})$'s   and the  fusion rule of 
 the CFT (= conformal field theory) for  the affine Lie algebra $\widehat{\mfs \mfl}_2$.
As a result of this comparison, 
 the Verlinde formula for that CFT yields  the following  explicit formula computing 
 $\mr{deg}(\Pi_{\mfs \mfl_2, g, r})$'s 
 for general $(g, r)$'s:
\begin{align}
\mr{deg}(\Pi_{\mfs \mfl_2, g, r})
=
  \frac{p^{g-1}}{2^{2g-1+r}} \cdot \sum_{j =1}^{p-1} 
\frac{\left(1-(-1)^j \cdot \cos \left(\frac{j\pi}{p} \right)\right)^r}{\sin^{2(g-1+r)} \left(\frac{j \pi}{p} \right)}
\end{align}
(cf. ~\cite[Theorem A]{Wak1}, ~\cite[Theorem 7.41]{Wak8}).
  \item[(c)]
Moreover,  based on the idea of K. Joshi et al. (cf., e.g., ~\cite{JP}, ~\cite{Jo14}),
the generic \'{e}taleness of $\Pi_{\mfs \mfl_n, g, 0}$ is applied   to lift relevant moduli spaces to characteristic $0$ and then compare them with certain 
Quot schemes
  over $\mbC$.
Hence,  
(under the assumption that $p$ is sufficiently large relative to $g$ and $n$)
the following formula for computing
  $\mr{deg}(\Pi_{\mfs \mfl_n, g, 0})$'s, originally conjectured by Joshi (cf. ~\cite[Conjecture 8.1]{Jo14}),
can be induced from a computation of the Gromov-Witten invariants of Grassmann varieties   (i.e., the Vafa-Intriligator formula) via 
 a work by Holla (cf. ~\cite{Hol}):

 \begin{align}
\mr{deg}(\Pi_{\mfs \mfl_n, g, 0})  = \frac{ p^{(n-1)(g-1)-1}}{n!} \cdot  \hspace{-5mm}
 \sum_{\genfrac{.}{.}{0pt}{}{(\zeta_1, \cdots, \zeta_n) \in \mbC^{\times n} }{ \zeta_i^p=1, \ \zeta_i \neq \zeta_j (i\neq j)}}
 \frac{(\prod_{i=1}^n\zeta_i)^{(n-1)(g-1)}}{\prod_{i\neq j}(\zeta_i -\zeta_j)^{g-1}}
 \end{align}
(cf.  ~\cite[Theorem H]{Wak8}).

\end{itemize}

As suggested above,
 the 
generic \'{e}taleness
 of $\Pi_{\mfg, g, r}$ has great importance
 from the viewpoint  of enumerative geometry.
 This property has also been verified 
  for $\mfg = \mfs \mfo_{2\ell -1}$ and $\mfs \mfp_{2m}$ (cf. ~\cite[Theorem G]{Wak8}), and we expect the same for general $\mfg$'s.
The purpose of  the present paper  is  to prove the generic \'{e}taleness for one of the remaining  cases, i.e., $\mfg = \mfs \mfo_{2\ell}$.
 Our  main result  is  described as follows.

\SSP
\begin{intthm}[cf. Theorem \ref{Thm1}] \label{ThA}
 Let $\ell$ be a positive integer with $\frac{p+2}{4} > \ell > 3$.
Then,
the stack  $\mcO p_{\mfs \mfo_{2 \ell}, g, r}^{^\mr{Zzz...}}$ 
is \'{e}tale over the points of 
$\overline{\mcM}_{g, r}$ classifying totally degenerate curves.
In particular, 
  $\mcO p_{\mfs \mfo_{2 \ell}, g, r}^{^\mr{Zzz...}}$ is generically \'{e}tale over $\overline{\mcM}_{g, r}$, i.e., any irreducible component that dominates $\overline{\mcM}_{g, r}$ admits a dense open substack which  is  \'{e}tale over $\overline{\mcM}_{g, r}$.
 \end{intthm}
\SSP

The above theorem 
makes progress toward realizing some correspondences for $\mfg = \mfs \mfo_{2\ell}$  as  displayed in the  above picture.
In fact,  by the resulting generic \'{e}taleness, we can apply  the previous study of (a kind of) fusion rings associated to 
the moduli space of dormant opers
 (cf. ~\cite[\S\,7]{Wak8}).
Let 
 $\mfc (\mbF_p)$ denote  the set of $\mbF_p$-rational points in the GIT quotient of $\mfs \mfo_{2\ell}$ by the action of  its  adjoint group.
Each element $\rho \in \mfc (\mbF_p)^{\times r} \left(= \mfc (\mbF_p) \times \cdots \times \mfc (\mbF_p) \right)$ determines  the  closed and  open substack
$\mcO p^{^\mr{Zzz...}}_{\mfs \mfo_{2\ell}, \rho, g, r}$
of $\mcO p_{\mfs \mfo_{2\ell}, g, r}^{^\mr{Zzz...}}$
classifying  dormant $\mfs \mfo_{2\ell}$-opers {\it of radii} $\rho$ (cf. ~\cite[Definition 2.32]{Wak8}).
Then,  $\mcO p^{^\mr{Zzz...}}_{\mfs \mfo_{2\ell}, g, r}$ decomposes into the disjoint union $\coprod_{\rho \in \mfc (\mbF_p)^{\times r}} \mcO p^{^\mr{Zzz...}}_{\mfs \mfo_{2\ell}, \rho, g, r}$,  which implies  a decomposition  of generic degrees $\mr{deg}(\Pi_{\mfs \mfo_{2\ell}, g, r}) = \sum_{\rho \in \mfc (\mbF_p)^{\times r}} \mr{deg}(\Pi_{\mfs \mfo_{2\ell}, \rho, g, r})$
under the assumption  in Theorem \ref{ThA}.

As a corollary of Theorem \ref{ThA}, we obtain an analogue  of the Verlinde formula
 computing the values $\mr{deg}(\Pi_{\mfs \mfo_{2\ell}, \rho, g, r})$, as follows.

\SSP
\begin{intthm}[cf. Theorem \ref{Thm15}] \label{ThC}
Let $\ell$ be a positive integer with $\frac{p+2}{4} > \ell > 3$.
Denote by $\Fus$ the pseudo-fusion ring for dormant $\mfs \mfo_{2\ell}$-opers with multiplication $\ast$ (cf. ~\cite[Definition 7.34]{Wak8}).
 Write $\mfS$ for the set of ring homomorphims $\Fus \rightarrow \mbC$ and write $\mr{Cas} := \sum_{\lambda \in \mfc (\mbF_p)} \lambda \ast \lambda \left(\in \Fus \right)$.
Then, for each $\rho := (\rho_i)_{i=1}^r \in \mfc (\mbF_p)^{\times r}$,
the following equality holds:
\begin{align}
\mr{deg}(\Pi_{\mfs \mfo_{2\ell}, \rho, g, r}) = \sum_{\chi \in \mfS} \chi (\mr{Cas})^{g-1} \cdot \prod_{i=1}^r \chi (\rho_i).
\end{align}
In particular, if $r = 0$ (which implies $g > 1$), then this equality reads
\begin{align}
\mr{deg}(\Pi_{\mfs \mfo_{2\ell}, \emptyset, g, 0}) = \sum_{\chi \in \mfS} \chi (\mr{Cas})^{g-1}.
\end{align}
 \end{intthm}
\SSP

Finally, we remark that, just as in the case of $\mfs \mfl_2$ (cf. ~\cite[\S\,7.8]{Wak8}),
the details   of 
$\mcO p^{^\mr{Zzz...}}_{\mfs \mfo_{2\ell}, \rho, 0, 3}$'s would be necessary to
  explicitly identify  $\mr{deg}(\Pi_{\mfs \mfo_{2\ell}, \rho, g, r})$'s because these values  characterize the ring-theoretic structure of $\Fus$.
  This is one of the issues to be addressed in future research  on the enumerative geometry of dormant opers.

\LSP
\subsection*{Notation and Conventions} 

All schemes appearing in the present paper are assumed to be locally noetherian.
Throughout our discussion, we fix 
an integer $\ell$ with $\ell \geq 2$
 and a pair of nonnegative integers $(g, r)$ with $2g-2 +r > 0$.
 Also, fix an algebraically closed field $k$ whose 
  characteristic $\mr{char}(k)$  satisfies 
  either ``$\mr{char}(k)= 0$" or ``$\mr{char}(k) > 2(2\ell -1)$".

Given a positive integer $n$ with $n \geq 2$, we denote by $\mr{GL}_n$ (resp., $\mr{PGL}_n$) the general (resp., projective) linear group of dimension  $n$ and by $\mr{GO}_n$ the group of orthogonal similitudes of dimension $n$.
For example, 
we have
\begin{align}
\mr{GO}_{2 \ell} := \left\{h \in \mr{GL}_{2 \ell} \, \Bigg| \,  \exists \, \mr{sim}(h) \in \mbG_m \ \text{s.t.} {^t h}\cdot \begin{pmatrix} O & E_\ell \\ E_\ell & O  \end{pmatrix} \cdot h = \mr{sim}(h) \cdot \begin{pmatrix} O & E_\ell \\ E_\ell & O\end{pmatrix}\right\},
\end{align}
where  $E_\ell$ denotes the $\ell \times \ell$ identity matrix.
Note that the group $\mr{GO}_{2\ell}$
  is disconnected, and its neutral component $\mr{GO}_{2\ell}^0 \subseteq \mr{GO}_{2\ell}$ is defined by the condition 
$\mr{det}(h) = \mr{sim}(h)^\ell$.
Denote by $\mr{PGO}_{2\ell}^0$ the adjoint group of $\mr{GO}_{2\ell}^0$ and fix a Borel subgroup $B$ of $\mr{PGO}_{2\ell}^0$.

Let $S^\mr{log}$ be an fs log scheme (cf.  ~\cite{KaKa}, ~\cite{ILL}, and ~\cite{FKa} for the basic properties and definitions  concerning  log schemes).
By a {\bf log curve} over $S^\mr{log}$,
we mean a log smooth integral morphism $f^\mr{log} : U^\mr{log} \rightarrow S^\mr{log}$ between fs log schemes such that every  geometric fiber of the underlying morphism  of schemes $f : U \rightarrow S$ is either empty or a reduced $1$-dimensional scheme (cf. ~\cite[Definition 1.40]{Wak8}).

Let us take a log curve $U^\mr{log}$ over an fs log scheme $S^\mr{log}$, and denote by $\Omega$ the sheaf of logarithmic $1$-forms on $U^\mr{log}$ over $S^\mr{log}$.
An {\bf $S^\mr{log}$-connection } on an $\mcO_U$-module $\mcF$ is  an $f^{-1}(\mcO_S)$-linear morphism $\nabla : \mcF \rightarrow \Omega \otimes \mcF$ satisfying the equality  $\nabla (a v) = d a \otimes v + a \cdot \nabla  (v)$ for any pair of  local sections $(a, v) \in \mcO_U \times  \mcF$ (cf.  ~\cite[Definition 4.1]{Wak8}).
(Since $U^\mr{log}/S^\mr{log}$ is a log curve, any $S^\mr{log}$-connection is automatically flat, in the sense of ~\cite[Definition 4.3]{Wak8}.)
By a {\bf flat vector bundle} on $U^\mr{log}/S^\mr{log}$, we mean a pair $(\mcF, \nabla)$ consisting of a vector bundle (i.e., a locally free coherent sheaf) $\mcF$ on $U$ and an $S^\mr{log}$-connection on it.

Next, denote by 
$\overline{\mcM}_{g, r}$
the moduli stack classifying $r$-pointed stable curves over $k$ of genus $g$ and by
\begin{align}
\msC_{g, r} := (f_{\mr{univ}} : \mcC_{g, r} \rightarrow \overline{\mcM}_{g, r}, \{ \sigma_{\mr{univ}, i} : \overline{\mcM}_{g, r} \rightarrow \mcC_{g, r} \}_{i=1}^r)
\end{align}
the universal family of $r$-pointed stable curves over $\overline{\mcM}_{g, r}$ (cf. ~\cite[Eq.\,(132)]{Wak8}), which consists of a prestable curve 
  $f_{\mr{univ}} : \mcC_{g, r} \rightarrow \overline{\mcM}_{g, r}$ over 
  $\overline{\mcM}_{g, r}$ and a collection of mutually disjoint $r$ marked points
  $\{ \sigma_{\mr{univ}, i} : \overline{\mcM}_{g, r} \rightarrow \mcC_{g, r} \}_{i=1}^r$ of that curve.
Recall from ~\cite[Theorem 4.5]{FKa} that  $\overline{\mcM}_{g, r}$  and $\mcC_{g, r}$ admit natural log structures; we denote the resulting fs log stack by $\overline{\mcM}_{g, r}^\mr{log}$  and $\mcC_{g, r}^\mr{log}$, respectively.

If we  take an $r$-pointed stable curve $\msX := (f: X \rightarrow S, \{ \sigma_i : S \rightarrow X \}_{i=1}^r)$ of genus $g$
  over a $k$-scheme $S$ (cf. ~\cite[Definition 1.38]{Wak8}),
then both $S$ and $X$ are equipped with log structures pulled-back from $\overline{\mcM}_{g, r}^\mr{log}$ and $\mcC_{g, r}^\mr{log}$, respectively,  via the cartesian square diagram
\begin{align} \label{Ed37}
\vcenter{\xymatrix@C=46pt@R=36pt{
 X \ar[r] \ar[d]_-{f} & \mcC_{g, r}\ar[d]^-{f_{\mr{univ}}} \\
 S \ar[r] & \overline{\mcM}_{g, r}
 }}
\end{align}
determined by the classifying morphism of $\msX$; we denote the resulting log structures by $S^\mr{log}$ and $X^\mr{log}$, respectively.
These are fs log schemes, and the morphism $f : X \rightarrow S$ extends to a log curve $f^\mr{log} : X^\mr{log} \rightarrow S^\mr{log}$.
In this way, each pointed stable curve yields a log curve.

\vspace{10mm}
\section{$\mr{GO}_{2\ell}^0$-opers on log curves} \label{S1}
\LSP

In this section, we introduce $\mr{GO}_{2\ell}^0$-opers on a log curve described in terms of vector bundles, and discuss the relationship with $\mr{GO}_{2\ell-1}$-opers (cf. Proposition \ref{Prop3}).
Our argument may be thought of  as a simple  generalization of the argument  in ~\cite[(c).\,2.9]{BeDr2}.

\LSP
\subsection{$\mr{GL}_n$-opers and $\mr{GO}_{2\ell -1}$-opers} \label{SS19}

To begin with, we recall from ~\cite{Wak8} the definitions of a $\mr{GL}_n$-oper and a $\mr{GO}_{2 \ell -1}$-oper  on a fixed  log curve.

Let us fix an fs log scheme $S^\mr{log}$ over $k$ and a log curve $U^\mr{log}$ over $S^\mr{log}$.
Denote by $\Omega$ the sheaf of logarithmic $1$-forms on $U^\mr{log}$ over $S^\mr{log}$ and by $\mcT$ its dual.
Also, for each $j  \in \mbZ \sqcup \{ \infty \}$,  we  denote by
$\mcD^{< j}$
  the sheaf of  logarithmic  crystalline differential operators of order $< j$ on $U^\mr{log}/S^\mr{log}$, i.e., the sheaf ``$\mcD_{\hslash, Y^\mr{log}/T^\mr{log}}^{< j}$" defined in ~\cite[\S\,4.2.1]{Wak8} such that the pair $(Y^\mr{log}/T^\mr{log}, \hslash)$ is taken to be $(U^\mr{log}/S^\mr{log}, 1)$.

Note that the sheaf $\mcD^{< j}$ admits two different structures of $\mcO_U$-module, i.e., one as given by left multiplication, where we denote this $\mcO_U$-module by ${^L}\mcD^{< j}$, and the other given by right multiplication, where we denote this $\mcO_U$-module by ${^R}\mcD^{< j}$.
Given an $\mcO_U$-module  $\mcF$,
we equip the tensor product $\mcD^{< j} \otimes \mcF := {^R}\mcD^{< j} \otimes \mcF$ with the $\mcO_U$-module structure
given by left multiplication.

Let 
$n$ be  a positive  integer, and 
 consider  a collection of data 
\begin{align} \label{Eq1}
\msF^\heartsuit := (\mcF, \nabla, \{ \mcF^j \}_{j=0}^n),
\end{align}
where
$\mcF$ denotes a vector bundle on $U$ of rank $n$,
$\nabla$ denotes an $S^\mr{log}$-connection  on $\mcF$,
and
$\{ \mcF^j \}_{j=0}^n$ denotes an $n$-step decreasing filtration 
\begin{align} \label{Eq357}
0 = \mcF^n \subseteq  \mcF^{n- 1} \subseteq \cdots \subseteq \mcF^0=\mcF
\end{align}
on $\mcF$ consisting of subbundles such that the subquotients $\mcF^j /\mcF^{j+1}$ ($j=0, \cdots, n-1$) are   line bundles.

\SSP
\bde[cf. ~\cite{Wak8}, Definition 4.17] \label{Def2}
\begin{itemize}
\item[(i)]
We shall say that $\msF^\heartsuit$ is a {\bf $\mr{GL}_n$-oper} on $U^\mr{log}/S^\mr{log}$ if it satisfies the following two conditions:
\begin{itemize}
\item
For each $j=1, \cdots, n-1$, $\nabla (\mcF^j)$ is  contained in $\Omega \otimes \mcF^{j-1}$;
\item
For each $j=1, \cdots, n-1$,
the well-defined $\mcO_U$-linear morphism 
\begin{align} \label{Eq356}
\mr{KS}^j
 : \mcF^j/\mcF^{j+1} \xrightarrow{} \Omega \otimes (\mcF^{j-1}/\mcF^j)
\end{align}
defined by $\overline{a} \mapsto \overline{\nabla (a)}$ for any local section $a \in \mcF^j$ (where $\overline{(-)}$'s denote the images in the respective quotients) is an isomorphism.
\end{itemize}
\item[(ii)]
Let $\msF^\heartsuit$ and $\msF'^\heartsuit$  be  $\mr{GL}_n$-opers on $U^\mr{log}/S^\mr{log}$.
An {\bf isomorphism of $\mr{GL}_n$-opers} from $\msF^\heartsuit$ to $\msF'^\heartsuit$ is an isomorphism between their  underlying vector bundles $\eta_\mcF :\mcF \xrightarrow{\sim} \mcF'$ preserving both the connection and  filtration.
\end{itemize}
\ede
\SSP

\begin{rem}[cf. ~\cite{Wak8}, Remark 4.19] \label{Rem}
Let $\msF^\heartsuit := (\mcF, \nabla, \{ \mcF^j \}_j)$ be a  $\mr{GL}_n$-oper on $U^\mr{log}/S^\mr{log}$.
The isomorphisms $\mr{KS}^j$ for various $j$'s together yield a composite isomorphism
\begin{align} \label{Eq53}
\mr{KS}^{j \Rightarrow n-1} 
: \mcF^j/\mcF^{j+1} &\xrightarrow{\sim} \mcT \otimes (\mcF^{j+1}/\mcF^{j+2})\xrightarrow{\sim} \cdots \\
\cdots &\xrightarrow{\sim} \mcT^{\otimes (n-1-j)} \otimes (\mcF^{n-1}/\mcF^n) \left(= \mcT^{\otimes (n-1-j)} \otimes \mcF^{n-1}\right). \notag
\end{align}
In particular, 
we have
\begin{align} \label{Eq98}
\mr{det}(\mcF) \xrightarrow{\sim} \bigotimes_{j=0}^{n-1} \mcF^j / \mcF^{j+1} \xrightarrow{\sim} \bigotimes_{j=0}^{n-1}  (\mcT^{\otimes (n-1-j)} \otimes \mcF^{n-1}) \xrightarrow{\sim} (\mcF^{n-1})^{\otimes n} \otimes \mcT^{\otimes \frac{n (n-1)}{2}}.
\end{align}
This  implies that the isomorphism class of  the determinant $\mr{det}(\mcF)$ depends only on $\mcF^{n-1}$.
\end{rem}
\SSP

Let $\mcF$ and $\mcG$ be vector bundles on $U$; we will not distinguish an $\mcO_U$-bilinear map $\mcF \times \mcF \rightarrow \mcG$
on $\mcF$ valued in $\mcG$ with  the corresponding $\mcO_U$-linear morphism $\mcF^{\otimes 2} \rightarrow \mcG$.
Given an $\mcO_U$-bilinear map $\omega : \mcF^{\otimes 2} \rightarrow \mcG$ on $\mcF$,
we write
\begin{align}
\omega_{(-, \bullet)} : \mcF \rightarrow \mcF^\vee \otimes \mcG
\end{align} 
for the $\mcO_U$-linear morphism given by $v \mapsto \omega (v \otimes (-))$
for any local section $v \in \mcF$.
Such a bilinear map  $\omega$ is called  {\bf nondegenerate} if ($\mcG$ has rank one and) $\omega_{(-, \bullet)}$ is an isomorphism.
If $\omega$ is nondegenerate, then the determinant of $\omega_{(-, \bullet)}$
determines an isomorphism
\begin{align} \label{Eq90}
\mr{det} (\omega) :  \mr{det}(\mcF)^{\otimes 2} \xrightarrow{\sim} \mcG^{\otimes \mr{rk}(\mcF)}.
 \end{align}

Now, let us  consider a collection of data
\begin{align} \label{eeWdis}
\msF^\heartsuit_{\sphericalangle} := (\mcF, \nabla, \{ \mcF^j \}_{j=0}^{2\ell -1}, \mcN, \nabla_\mcN, \omega),
\end{align}
where $(\mcF, \nabla, \{ \mcF^j \}_{j=0}^{2\ell -1})$ forms a $\mr{GL}_{2 \ell -1}$-oper on $U^\mr{log}/S^\mr{log}$,
$(\mcN, \nabla_\mcN)$ denotes a flat line bundle on $U^\mr{log}/S^\mr{log}$ (i.e.,
a flat vector bundle such that  $\mcN$ has rank one),
and $\omega$ denotes a nondegenerate {\it symmetric} $\mcO_U$-bilinear map $\mcF^{\otimes 2} \rightarrow \mcN$ on $\mcF$ valued in $\mcN$.

\SSP
\bde[cf. ~\cite{Wak8}, Definition 5.1] \label{Def9}
\begin{itemize}
\item[(i)]
Keeping the above notation, we  say that 
$\msF^\heartsuit_{\sphericalangle}$ is a {\bf $\mr{GO}_{2 \ell -1}$-oper} on $U^\mr{log}/S^\mr{log}$ if it satisfies the following two conditions:
\begin{itemize}
\item
The $S^\mr{log}$-connection $\nabla^{\otimes 2}$ on $\mcF^{\otimes 2}$ induced naturally from $\nabla$ is compatible with $\nabla_\mcN$  via 
$\omega$;
\item
For any $j= 0, \cdots, 2\ell -1$,
the equality $\mcF^{2\ell -1-j} = (\mcF^{j})^\perp \left(:= \mr{Ker}(\omega_{(-, \bullet)} |_{\mcF^j}) \right)$ holds.
\end{itemize}
\item[(ii)]
Let $\msF^\heartsuit_{\sphericalangle}$ and $\msF'^\heartsuit_{\sphericalangle}$ be $\mr{GO}_{2 \ell -1}$-opers on $U^\mr{log}/S^\mr{log}$.
An {\bf isomorphism of $\mr{GO}_{2\ell -1}$-opers} from 
$\msF^\heartsuit_{\sphericalangle}$ to $\msF'^\heartsuit_{\sphericalangle}$
is defined as a pair
\begin{align}
(\eta_\mcF, \eta_\mcN)
\end{align}
consisting of an isomorphism  between their respective  underlying $\mr{GL}_{2 \ell -1}$-opers $\eta_\mcF : \msF^\heartsuit \xrightarrow{\sim} \msF'^\heartsuit$  and an isomorphism $\eta_\mcN : (\mcN, \nabla_\mcN) \xrightarrow{\sim} (\mcN', \nabla_{\mcN'})$ of flat line bundles satisfying $\eta_\mcN \circ \omega = \omega' \circ \eta_\mcF^{\otimes 2}$. 

\end{itemize}
\ede
\SSP

\begin{rem} \label{Eq55}
Let $\msF^\heartsuit_{\sphericalangle}$ be a $\mr{GO}_{2\ell -1}$-oper  as in \eqref{eeWdis}.
The bilinear map $\omega$ induces an isomorphism
$(\mcF^{\ell -1}/\mcF^\ell)^{\otimes 2} \xrightarrow{\sim} \mcN$.
On the other hand, it follows from the argument of Remark \ref{Eq53} that there exists an isomorphism 
$\mcF^{\ell -1}/\mcF^\ell \xrightarrow{\sim} \mcT^{\otimes (\ell -1)} \otimes \mcF^{2\ell -2}$.
By composing these isomorphisms, we obtain an isomorphism
\begin{align} \label{Eq1255}
( \mcT^{\otimes (\ell -1)} \otimes \mcF^{2\ell -2})^{\otimes 2} \xrightarrow{\sim} \mcN. 
\end{align}
By putting $\overline{\mcF}:= \mr{det}(\mcF)^\vee \otimes \mcN^{\otimes \ell}$,
we obtain a composite isomorphism
\begin{align} \label{Eq127}
\overline{\mcF}^{\otimes 2} 
&\xrightarrow{\sim} \mr{det}(\mcF)^{\otimes (-2)} \otimes \mcN^{\otimes 2\ell} \\ &\xrightarrow{\sim} 
((\mcF^{2\ell -2})^{\otimes -2 (2\ell -1)} \otimes \mcT^{\otimes -2(2\ell -1)(\ell -1)})\otimes \mcN^{\otimes 2\ell}  \notag \\
&\xrightarrow{\sim} 
\mcN^{\otimes (-2 \ell +1)} \otimes \mcN^{\otimes 2\ell}  \notag \\
&\xrightarrow{\sim} \mcN,
\end{align}
where the second and third  arrows follow from \eqref{Eq98} and  
\eqref{Eq1255}, respectively.
\end{rem}

\LSP
\subsection{$\mr{GO}_{2 \ell}^0$-opers} \label{SS550}

Next, let us consider a collection of data
\begin{align} \label{Eq38}
\msF^\heartsuit_{\sphericalangle, +} := (\mcF_+, \nabla_+, \{ \mcF^j_+\}_{j=0}^{2 \ell -1}, \mcN, \nabla_\mcN, \omega_+, \varpi),
\end{align}
where
\begin{itemize}
\item
$(\mcF_+, \nabla_+)$ is a rank $2\ell$ flat vector bundle on $U^\mr{log}/S^\mr{log}$;
\item
$\{ \mcF^j_+ \}_{j=0}^{2\ell -1}$ is a $(2 \ell -1)$-step decreasing filtration 
\begin{align}
0 = \mcF^{2 \ell -1}_+ \subseteq  \mcF^{2 \ell - 2}_+ \subseteq \cdots \subseteq \mcF^0_+=\mcF_+
\end{align}
on $\mcF_+$ consisting of subbundles such that $\mcF^j_+/\mcF^{j+1}_+$ has rank one   for $j \neq \ell -1$ and $\mcF^{\ell -1}_+/\mcF^\ell_+$ has  rank $2$;
\item
$(\mcN, \nabla_\mcN)$ is a flat line bundle on $U^\mr{log}/S^\mr{log}$;
\item
$\omega_+$ is a nondegenerate $\mcO_U$-bilinear map $\mcF^{\otimes 2}_+ \rightarrow \mcN$ on $\mcF_+$ such that  $\mcF^{2 \ell - j-1}_+ = (\mcF^j_+)^\perp$
  for every  $j=0, \cdots, 2 \ell -1$ and that $\nabla^{\otimes 2}_+$ is compatible with $\nabla_\mcN$ via $\omega_+$ (i.e., it specifies a morphism of flat vector  bundles $(\mcF_+, \nabla_+)^{\otimes 2}\rightarrow (\mcN, \nabla_\mcN)$);
\item
$\varpi$ is an isomorphism $(\mr{det}(\mcF_+), \mr{det}(\nabla_+)) \xrightarrow{\sim} (\mcN, \nabla_\mcN)^{\otimes \ell}$ satisfying the equality $\varpi^{\otimes 2} = \mr{det}(\omega_+)$ (cf. \eqref{Eq90}).
\end{itemize}

\SSP
\bde \label{Def5}
\begin{itemize}
\item[(i)]
We say that $\msF^\heartsuit_{\sphericalangle, +}$ is a {\bf $\mr{GO}^0_{2 \ell}$-oper} on $U^\mr{log}/S^\mr{log}$ if it satisfies  the following three  conditions:
\begin{itemize}
\item
For each $j = 1, \cdots, 2\ell-2$,
$\nabla_+ (\mcF^j_+)$ is contained in $\Omega \otimes \mcF_+^{j-1}$;
\item
For each $j=1, \cdots, 2\ell -2$ with $j \neq \ell -1, \ell$,
the well-defined $\mcO_U$-linear morphism
\begin{align}
\mr{KS}^j
 : \mcF^j_+ /\mcF^{j+1}_+ \xrightarrow{} \Omega \otimes (\mcF^{j-1}_+/\mcF^j_+)  
\end{align}
is an isomorphism;
\item
The composite  morphism
\begin{align} \label{Eq46}
\mcF^{\ell}_+/\mcF_+^{\ell +1} \xrightarrow{\mr{KS}^\ell} \Omega \otimes (\mcF^{\ell -1}_+/\mcF_+^{\ell}) \xrightarrow{\mr{id} \otimes \mr{KS}^{\ell -1}} \Omega^{\otimes 2} \otimes (\mcF_+^{\ell -2}/\mcF_+^{\ell -1})
\end{align}
 is an isomorphism.
\end{itemize}
\item[(ii)]
Let $\msF^\heartsuit_{\sphericalangle, +}$ and $\msF'^\heartsuit_{\sphericalangle, +}$ be $\mr{GO}_{2 \ell }^0$-opers on $U^\mr{log}/S^\mr{log}$ (as in \eqref{Eq38}).
An {\bf isomorphism of $\mr{GO}_{2\ell }^0$-opers} from 
$\msF^\heartsuit_{\sphericalangle, +}$ to $\msF'^\heartsuit_{\sphericalangle, +}$
is defined as a pair
\begin{align}
\eta := (\eta_{\mcF_+}, \eta_\mcN)
\end{align}
consisting of an isomorphism  between their respective  underlying 
vector bundles $\eta_{\mcF_+} : \mcF_+ \xrightarrow{\sim} \mcF_+$
preserving both the filtration and connection, 
  and an isomorphism  of flat line bundles $\eta_\mcN : (\mcN, \nabla_\mcN) \xrightarrow{\sim} (\mcN', \nabla_{\mcN'})$  
such that the following square diagrams are commutative:
\begin{align} \label{Eq49}
\vcenter{\xymatrix@C=46pt@R=36pt{
\mcF^{\otimes 2}_+ \ar[r]^-{\eta_{\mcF_+}^{\otimes 2}}_-{\sim} \ar[d]_-{\omega_+} & \mcF'^{\otimes 2}_+  \ar[d]^-{\omega'_+} 
\\
\mcN \ar[r]_-{\eta_\mcN}^-{\sim} & \mcN',
 }}
 \hspace{15mm}
 \vcenter{\xymatrix@C=46pt@R=36pt{
\mr{det} (\mcF_+) \ar[r]^-{\mr{det}(\eta_{\mcF_+})}_-{\sim} \ar[d]_-{\varpi}^-{\wr} & \mr{det}(\mcF'_+) \ar[d]^-{\varpi'}_-{\wr}
\\
\mcN^{\otimes \ell} \ar[r]_-{\eta_\mcN^{\otimes \ell}}^-{\sim} & \mcN'^{\otimes \ell}.
 }}
\end{align}

\end{itemize}
\ede
\SSP

\subsection{The relation between  $\mr{GO}_{2 \ell }^0$-opers and $\mr{GO}_{2 \ell -1}$-opers} \label{SS23}

Let us take a $\mr{GO}_{2 \ell}^0$-oper $\msF^{\heartsuit}_{\sphericalangle, +} := (\mcF_+, \nabla_+, \{ \mcF^j_+ \}_j, \mcN, \nabla_{\mcN}, \omega_+, \varpi)$  on $U^\mr{log}/S^\mr{log}$.
By equipping $\mcF_+$ with  a $\mcD^{< \infty}$-module structure determined by $\nabla_+$,
we obtain 
 the $\mcD^{< \infty}$-submodule $\mcF$ of $\mcF_+$ generated by the local sections of  the line subbundle $\mcF^{2 \ell -2}_+$.
If $\nabla$  denotes the $S^\mr{log}$-connection on $\mcF$ obtained by restricting $\nabla_+$, then 
the pair  $(\mcF, \nabla)$
form a rank $(2\ell -1)$ flat subbundle of $(\mcF_+, \nabla_+)$. 

We shall write  $\overline{\mcF} := \mcF_+/\mcF$ and  write  $\overline{\nabla}$ for
 the $S^\mr{log}$-connection  on 
 $\overline{\mcF}$  induced from $\nabla_+$ via the quotient $\mcF_+ \twoheadrightarrow \overline{\mcF}$.
The pair $(\overline{\mcF}, \overline{\nabla})$ specifies a flat line bundle, which fits into the following  short exact sequence of flat vector  bundles:
 \begin{align} \label{Eq51}
 0 \rightarrow (\mcF, \nabla) \rightarrow (\mcF_+, \nabla_+) \rightarrow (\overline{\mcF}, \overline{\nabla}) \rightarrow 0.
 \end{align}
 The induced isomorphism $\mr{det}(\mcF) \otimes \overline{\mcF} \xrightarrow{\sim}\mr{det}(\mcF_+)$ gives a composite isomorphism
\begin{align} \label{Eq91}
\overline{\mcF} \xrightarrow{\sim} \mr{det}(\mcF)^\vee \otimes \mr{det}(\mcF_+) \xrightarrow{\mr{id} \otimes \varpi} \mr{det}(\mcF)^\vee \otimes \mcN^{\otimes \ell},
\end{align}
by which we often identify $\overline{\mcF}$ with $\mr{det}(\mcF)^\vee \otimes \mcN^{\otimes \ell}$.

Since the composite
\begin{align}
\mcF \xrightarrow{\mr{inclusion}} \mcF_+\xrightarrow{(\omega_+)_{(-, \bullet)}} \mcF^\vee_+ \otimes \mcN \xrightarrow{\mr{quotient}} \mcF^\vee \otimes \mcN 
\end{align}
is an isomorphism,  it
 determines a decomposition
\begin{align} \label{Eq52}
\mcF_+ = \mcF \oplus \overline{\mcF}.
\end{align}
Also,   $\omega_+$ induces, via this decomposition,  nondegenerate  bilinear maps
\begin{align}
\omega: (\mcF, \nabla)^{\otimes 2} \rightarrow (\mcN, \nabla_\mcN)
\hspace{3mm} \text{and} \hspace{3mm}
\overline{\omega} : (\overline{\mcF}, \overline{\nabla})^{\otimes 2} \xrightarrow{\sim} (\mcN, \nabla_\mcN)
\end{align}
 on $\mcF$ and $\overline{\mcF}$, respectively.
 The underlying morphism between  line bundles of $\overline{\omega}$ coincides with \eqref{Eq127}.
 By putting 
  $\mcF^j := \mcF\cap \mcF^j_+$ ($j=0, \cdots, 2 \ell -1$),
  we obtain a collection of data
\begin{align} \label{Eq110}
\msF^{\heartsuit}_{\sphericalangle, + \Rightarrow \emptyset} := (\mcF, \nabla, \{ \mcF^j \}_{j=0}^{2\ell -1}, \mcN, \nabla_\mcN, \omega),
\end{align}
which forms a $\mr{GO}_{2\ell -1}$-oper on $U^\mr{log}/S^\mr{log}$.

The short exact sequence \eqref{Eq51}  implies that,
under the identification $\mcF_+ = \mcF \oplus \overline{\mcF}$ given by  \eqref{Eq52},  
$\nabla_+$ may be expressed   as the  sum $(\nabla \oplus  \overline{\nabla}) + \nu (\msF^{\heartsuit}_{\sphericalangle, +}) : \mcF \oplus \overline{\mcF} \rightarrow \Omega \otimes (\mcF \oplus \overline{\mcF})$ for a unique  $\mcO_U$-linear morphism 
\begin{align} \label{Eq60}
\nu (\msF^{\heartsuit}_{\sphericalangle, +}) : 
\overline{\mcF} \rightarrow \Omega \otimes \mcF.
\end{align}
In this way,  each $\mr{GO}_{2\ell}$-oper  $\msF^{\heartsuit}_{\sphericalangle, +}$ determines 
a pair of data
\begin{align} \label{Eq124}
(\msF^{\heartsuit}_{\sphericalangle, + \Rightarrow \emptyset}, \nu (\msF^{\heartsuit}_{\sphericalangle, +})).
\end{align}

Next, let
$\msF^{\heartsuit}_{\sphericalangle +}$ and $\msF'^{\heartsuit}_{\sphericalangle, +}$ 
 be $\mr{GO}_{2\ell}^0$-opers on $U^\mr{log}/S^\mr{log}$ as in \eqref{Eq38}  and 
 $\eta := (\eta_{\mcF_+}, \eta_\mcN) : \msF^{\heartsuit}_{\sphericalangle, +} \xrightarrow{\sim} \msF'^{\heartsuit}_{\sphericalangle, +}$ an isomorphism of $\mr{GO}_{2\ell}^0$-opers.
Then,
 $\eta_{\mcF_+}$
 restricts to an isomorphism $\eta_\mcF : \mcF \xrightarrow{\sim} \mcF'$, forming an isomorphism of $\mr{GO}_{2 \ell -1}$-opers $\msF^{\heartsuit}_{\sphericalangle, +\Rightarrow \emptyset} \xrightarrow{\sim} \msF'^{\heartsuit}_{\sphericalangle, +\Rightarrow \emptyset}$.
 Also, if 
   $\eta_{\overline{\mcF}}: (\overline{\mcF}, \overline{\nabla}) \xrightarrow{\sim} (\overline{\mcF}', \overline{\nabla}')$  denotes the isomorphism of flat line bundles induced from
    $\eta_{\mcF_+}$ via taking quotients, 
    then
   it satisfies the equality 
   $\eta_{\overline{\mcF}} = \mr{det}(\eta_\mcF)^\vee \otimes \eta_\mcN^{\otimes \ell}$ via \eqref{Eq91}  
  and
   fits into the following isomorphism of short exact sequences
\begin{align} \label{Eq49}
\vcenter{\xymatrix@C=46pt@R=36pt{
0 \ar[r] &  (\mcF, \nabla) \ar[r]^-{\mr{inclusion}} \ar[d]_{\wr}^-{\eta_{\mcF}} & (\mcF_+, \nabla_+) \ar[r]^-{\mr{quotient}} \ar[d]_{\wr}^-{\eta_{\mcF_+}} & (\overline{\mcF},  \overline{\nabla}) \ar[r] \ar[d]_{\wr}^{\eta_{\overline{\mcF}}} & 0 \\
0 \ar[r] &  (\mcF', \nabla') \ar[r]_-{\mr{inclusion}} & (\mcF'_+, \nabla'_+) \ar[r]_-{\mr{quotient}} & (\overline{\mcF}', \overline{\nabla}') \ar[r] & 0. 
 }}
\end{align}
Under the identifications $\mcF_+ = \mcF \oplus \overline{\mcF}$, $\mcF'_+ = \mcF' \oplus \overline{\mcF}'$ given by \eqref{Eq52},
the isomorphism $\eta_{\mcF_+}$ may be expressed  as the sum $(\eta_\mcF \oplus  \eta_{\overline{\mcF}}) + \nu (\eta) : \mcF \oplus \overline{\mcF} \xrightarrow{\sim} \mcF' \oplus \overline{\mcF}'$ for some $\mcO_U$-linear morphism
 \begin{align}
\eta_\nu : \overline{\mcF} \rightarrow \mcF'.
 \end{align}
Thus,  $\eta$ determines  a triple  of data
\begin{align} \label{Eq123}
(\eta_\mcF, \eta_\mcN, \eta_\nu).
\end{align}

\SSP
\ble \label{Lem3}
Let us keep the above notation.
Then,
the following equality of morphisms  $\overline{\mcF} \rightarrow \Omega \otimes \mcF'$ holds:
\begin{align}
\eta_\mcF \circ \nu (\msF^{\heartsuit}_{\sphericalangle, +}) -\nu (\msF'^{\heartsuit}_{\sphericalangle, +}) \circ \eta_{\overline{\mcF}}  = (\overline{\nabla}^\vee \otimes \nabla')(\eta_\nu), 
\end{align}
where $\overline{\nabla}^\vee \otimes \nabla'$ denotes the $S^\mr{log}$-connection on $\mcH om_{\mcO_X} (\overline{\mcF}, \mcF') \left(=  \overline{\mcF}^\vee \otimes \mcF'\right)$ induced naturally from  $\nabla'$ and (the dual of) $\overline{\nabla}$.
\ele
\begin{proof}
Since the problem is of local nature, we may assume
 that $\mcT= \mcO_U \partial$ for some section $\partial \in \Gamma (U, \mcT)$ (viewed  as a derivation on $\mcO_U$) and 
various    morphisms involved  are described 
 as
\begin{align}
\nabla = \partial +  A, 
\hspace{10mm}
\nabla' = \partial +  A',
\hspace{10mm}
 \overline{\nabla} = \partial + a, 
 \hspace{10mm}
 \overline{\nabla}'= \partial + a'
\end{align}
for some $A \in \mr{End}_{\mcO_U}(\mcF)$, $A' \in \mr{End}_{\mcO_U}(\mcF')$, $a \in \mr{End}_{\mcO_U} (\overline{\mcF})$, and  $a' \in \mr{End}_{\mcO_U} (\overline{\mcF}')$.
Since $\eta_{\mcF_+}$ preserves the connection, 
 we obtain an equality 
\begin{align}
\left(\partial +  \begin{pmatrix}  A' & \nu (\msF'^{\heartsuit}_{\sphericalangle, +})    \\
0  & a' \end{pmatrix}\right)  \circ \begin{pmatrix} \eta_\mcF  & \eta_\nu   \\ 0 & \eta_{\overline{\mcF}}\end{pmatrix}
=  \begin{pmatrix} \eta_\mcF  & \eta_\nu  \\ 0 & \eta_{\overline{\mcF}}\end{pmatrix} \circ \left(\partial +  \begin{pmatrix}  A & \nu (\msF^{\heartsuit}_{\sphericalangle, +})    \\
0  & a\end{pmatrix} \right)
\end{align}
of morphisms $\mcF \oplus \overline{\mcF} \rightarrow \Omega \otimes (\mcF' \oplus \overline{\mcF}')$. 
The $(1, 2)$-component  of this equality reads
\begin{align}
\bcancel{\eta_\nu \circ \partial }+ \partial (\nu ( \eta)) + A' \circ \eta_\nu + \nu (\msF'^{\heartsuit}_{\sphericalangle, +}) \circ \eta_{\overline{\mcF}}  
=\bcancel{ \eta_\nu \circ  \partial} + \eta_\mcF \circ  \nu (\msF^{\heartsuit}_{\sphericalangle, +} )  + \eta_\nu\circ a.
\end{align}
This is nothing but  the desired equality.
\end{proof}
\SSP

We shall denote by
\begin{align} \label{Eq941}
\mcO p_{2\ell} (U^\mr{log})
\end{align}
the groupoids consisting of  $\mr{GO}_{2 \ell}^0$-opers on $U^\mr{log}/S^\mr{log}$ and isomorphisms between them.
Also, we denote by 
\begin{align} \label{Eq100}
\mcO p_{2 \ell -1}^+ (U^\mr{log})
\end{align}
 the groupoid defined as follows:
\begin{itemize}
\item
 The objects are pairs  $(\msF^{\heartsuit}_\sphericalangle, \nu)$, where
\begin{itemize}
\item
$\msF^{\heartsuit}_\sphericalangle := (\mcF, \nabla, \{ \mcF^j \}_j, \mcN, \nabla_\mcN, \omega)$ denotes a $\mr{GO}_{2 \ell -1}$-oper on $U^\mr{log}/S^\mr{log}$;
\item
$\nu$ denotes an $\mcO_U$-linear morphism $\overline{\mcF} \rightarrow \Omega \otimes \mcF$
(where $\overline{\mcF} := \mr{det}(\mcF)^\vee \otimes \mcN^{\otimes \ell}$).
\end{itemize}
\item
The morphisms from $(\msF^{\heartsuit}_\sphericalangle, \nu)$ to $(\msF'^{\heartsuit}_\sphericalangle, \nu')$ are collections
$(\eta_\mcF, \eta_\mcN, \eta_\nu)$, where
\begin{itemize}
\item
$(\eta_\mcF, \eta_\mcN)$ is 
an isomorphism of $\mr{GO}_{2 \ell -1}$-opers  $\msF^{\heartsuit}_\sphericalangle \xrightarrow{\sim} \msF'^{\heartsuit}_\sphericalangle$;
\item
$\eta_\nu$ denotes 
an $\mcO_U$-linear morphism $\overline{\mcF} \rightarrow \mcF'$ satisfying the equality
$\eta_\mcF \circ \nu - \nu' \circ \eta_{\overline{\mcF}} = (\overline{\nabla}^\vee \otimes \nabla') (\eta_\nu)$, where $(\overline{\mcF}, \overline{\nabla})$ and $(\mcF', \nabla')$ are flat vector bundles determined (in the above manner) by $\msF^{\heartsuit}_\sphericalangle$ and $\msF'^{\heartsuit}_\sphericalangle$, respectively, and $\eta_{\overline{\mcF}} := \mr{det}(\eta_\mcF)^\vee \otimes \eta_\mcN^{\otimes \ell}$.
\end{itemize} 
\end{itemize}

\SSP
\bpr \label{Prop3}
The assignments
$\msF^{\heartsuit}_{\sphericalangle, +} \mapsto (\msF^{\heartsuit}_{\sphericalangle, + \Rightarrow \emptyset}, \nu (\msF^{\heartsuit}_{\sphericalangle, +}))$
and $\eta \mapsto (\eta_\mcF, \eta_\mcN, \eta_\nu)$ constructed in \eqref{Eq124} and \eqref{Eq123}, respectively, 
determines 
   an equivalence of categories
\begin{align} \label{Eq68}
\mcO p_{2 \ell} (U^\mr{log}) \xrightarrow{\sim} \mcO p_{2 \ell -1}^+ (U^\mr{log}).
\end{align}
Moreover, the formation of this equivalence commutes with pull-back by any \'{e}tale $U$-scheme (equipped with the natural log structure pulled-back from that on $U^\mr{log}$, which gives a structure of log curve over $S^\mr{log}$), as well as with base-change to any fs log scheme over  $S^\mr{log}$.
\epr
\begin{proof}
Let us take an object $(\msF^\heartsuit_{\sphericalangle}, \nu)$ of $\mcO p_{2 \ell -1}^+ (U^\mr{log})$, where $\msF^\heartsuit_{\sphericalangle} := (\mcF, \nabla, \{ \mcF^j \}_{j=0}^{2\ell -1}, \mcN, \nabla_\mcN, \omega)$.
We shall set $\overline{\mcF} := \mr{det}(\mcF)^\vee \otimes \mcN^{\otimes \ell}$,   $\mcF_+ := \mcF \oplus \overline{\mcF}$, and set $\mcF_+^j := \mcF^j \oplus \overline{\mcF}$ (resp., $\mcF_{+}^j := \mcF^j$) if $j = 0, \cdots, \ell -1$ (resp., $j = \ell, \cdots, 2\ell -1$).
Also, set $\nabla_+ := (\nabla \oplus d) + \nu$, which specifies an $S^\mr{log}$-connection on $\mcF_+$.
There exists a unique bilinear map  $\omega_+ : \mcF_+^{\otimes 2}\rightarrow \mcN$ such that
 the decomposition $\mcF_+ =\mcF \oplus \overline{\mcF}$ is orthogonal and its restriction to $\mcF$ (resp., $\overline{\mcF}$) coincides with $\omega$ (resp., \eqref{Eq127}).
Observe that the composite isomorphism
 \begin{align}
\varpi :  \mr{det}(\mcF_+) \xrightarrow{\sim} \mr{det} (\mcF) \otimes \overline{\mcF} 
 \xrightarrow{\sim} \mcN^{\otimes \ell}
 \end{align}
 induced by $\overline{\mcF} = \mr{det}(\mcF)^\vee \otimes \mcN^{\otimes \ell}$  is  compatible with the $S^\mr{log}$-connections $\mr{det}(\nabla_+)$ and $\nabla_\mcN^{\otimes \ell}$. 
 Thus, the  resulting collection
 \begin{align}
 \msF^\heartsuit_{\sphericalangle, +}:= (\mcF_+, \nabla_+, \{ \mcF_+^j \}_j, \mcN, \nabla_\mcN, \omega_+, \varpi)
 \end{align}
forms a $\mr{GO}_{2\ell}^0$-oper.
The assignment  $(\msF^\heartsuit_{\sphericalangle}, \nu) \mapsto \msF^\heartsuit_{\sphericalangle, +}$
 turns out to define an inverse of the assignment
$\msF^{\heartsuit}_{\sphericalangle, +} \mapsto (\msF^{\heartsuit}_{\sphericalangle, + \Rightarrow \emptyset}, \nu (\msF^{\heartsuit}_{\sphericalangle, +}))$, so
 we obtain the desired equivalence of categories \eqref{Eq68}.
\end{proof}

\vspace{10mm}
\section{$(\mr{GO}_{2\ell}^0, \vartheta)$-opers on log curves} \label{S154}
\LSP

This section discusses $\mr{GO}_{2\ell}^0$-opers whose determinants are fixed by using  a kind of generalized theta characteristic (i.e., a $(2\ell-1)$-theta characteristic).
We will give a bijective correspondence with $\mfs \mfo_{2\ell}$-opers (cf. Proposition \ref{Prop28}), via  which these objects can be described    in terms of  $\mfs \mfo_{2\ell -1}$-oper (cf. Theorem-Definition  \ref{Cor4}).

Let 
keep the notation introduced at the beginning of \S\,\ref{SS19}.

\LSP
\subsection{$(\mr{GO}_{2 \ell}^0, \vartheta)$-opers} \label{SS10}

Let $\vartheta := (\varTheta, \nabla_\vartheta)$ be a $(2\ell -1)$-theta  characteristic of $U^\mr{log}/S^\mr{log}$ in the sense of 
~\cite[Definition 4.31]{Wak8}, i.e., 
a pair consisting of a line bundle $\varTheta$ on $U$ and an $S^\mr{log}$-connection $\nabla_\vartheta$ on the line bundle $\mcT^{(2\ell -1)(\ell -1)} \otimes \varTheta^{\otimes (2 \ell -1)}$.
In the subsequent discussion, we abuse notation by writing $\vartheta$ for its pull-backs by \'{e}tale $U$-schemes,  as well as its base-changes to fs log schemes over $S^\mr{log}$ (cf. ~\cite[\S\,4.6.2]{Wak8}).

We shall write $\overline{\mcF}_\varTheta := \mcT^{\otimes (\ell -1)} \otimes \varTheta$, $\mcN_\varTheta :=   (\mcT^{\otimes (\ell -1)} \otimes \varTheta)^{\otimes 2}$, $\mcF_\varTheta := \mcD^{< (2 \ell -1)} \otimes \varTheta$, and 
\begin{align}
\mcF^j_\varTheta := \mcD^{< (2\ell -j-1)} \otimes \varTheta \  \ (j=0, \cdots, 2 \ell -1).
\end{align}
Note that $\{ \mcF_\varTheta^j \}_{j=0}^{2\ell -1}$ forms a $(2 \ell -1)$-step decreasing filtration on the rank $(2\ell -1)$ vector bundle $\mcF_\varTheta$, and 
the  subquotient $\mcF_\varTheta^j / \mcF_\varTheta^{j+1}$ (for each $j= 0, \cdots, 2\ell -2$) is naturally isomorphic to $\mcT^{\otimes (2 \ell -j -2)} \otimes \varTheta$.
Hence,
we obtain the composite of canonical isomorphisms
\begin{align} \label{Eq131}
\mr{det}(\mcF_\varTheta)  \xrightarrow{\sim} \bigotimes_{j=0}^{2\ell -2} (\mcF_\varTheta^j/\mcF_\varTheta^{j+1})
\xrightarrow{\sim}
\bigotimes_{j=0}^{2\ell -2} (\mcT^{\otimes (2\ell - j -2)}\otimes \varTheta)
\xrightarrow{\sim} \mcT^{\otimes (2\ell -1)(\ell -1)} \otimes \varTheta^{\otimes (2\ell -1)},
\end{align}
which induces 
\begin{align} \label{Eq130}
\overline{\mcF}_\varTheta
\xrightarrow{\sim}
(\mcT^{\otimes (2\ell -1)(\ell -1)} \otimes \varTheta^{\otimes (2\ell -1)})^\vee \otimes  (\mcT^{\otimes (\ell -1)} \otimes \varTheta)^{\otimes 2\ell}
\xrightarrow{\sim}\mr{det}(\mcF_\varTheta)^\vee \otimes \mcN^{\otimes \ell}_\varTheta.
\end{align}

Next, we shall set
$\mcF_{\varTheta, +} := \mcF_\varTheta \oplus \overline{\mcF}_\varTheta$
and 
\begin{align}
\mcF_{\varTheta, +}^j := \begin{cases} \mcF^j_\varTheta \oplus \overline{\mcF}_\varTheta& \text{if $0 \leq j\leq \ell -1$}   \\ \mcF^j_\varTheta & \text{if $\ell \leq j \leq 2 \ell -1$}.\end{cases}
\end{align}
It follows from ~\cite[Proposition 4.22, (i)]{Wak8} that there exists a unique pair of 
$S^\mr{log}$-connections
\begin{align}
\nabla_{\mcN_\varTheta} : \mcN_\varTheta \rightarrow \Omega\otimes \mcN_\varTheta, \hspace{5mm}
\nabla_{\overline{\mcF}_\varTheta} : \overline{\mcF}_\varTheta \rightarrow \Omega \otimes \overline{\mcF}_\varTheta
\end{align}
such that
  $\nabla_{\mcN_\varTheta}^{\otimes (2\ell -1)} = \nabla_\vartheta^{\otimes 2}$ 
  and $\nabla_{\overline{\mcF}_\varTheta}^{\otimes 2} = \nabla_{\mcN_\varTheta}$ 
  under  natural identifications $\mcN_{\varTheta}^{\otimes (2\ell -1)} = (\mcT^{\otimes \frac{n(n-1)}{2}} \otimes \varTheta^{\otimes n})^{\otimes 2}$ and $\overline{\mcF}_\varTheta^{\otimes 2} = \mcN_\varTheta$, respectively.
Also,  we obtain 
\begin{align}
\varpi_\varTheta : \mr{det}(\mcF_{\varTheta, +}) \xrightarrow{\sim} \mr{det}(\mcF_\varTheta) \otimes \overline{\mcF}_\varTheta \xrightarrow{\sim}
(\mcT^{\otimes (2\ell -1)(\ell -1)} \otimes \varTheta^{\otimes (2\ell -1)}) \otimes (\mcT^{\otimes (\ell -1)} \otimes \varTheta) \xrightarrow{\sim} \mcN_\varTheta^{\otimes \ell}.
\end{align}
The non-resp'd portion of the following definition was  already discussed in ~\cite[Definitions 4.36 and 5.4]{Wak8}.

\SSP
\bde \label{Def7}
\begin{itemize}
\item[(i)]
By a {\bf $(\mr{GO}_{2\ell -1}, \vartheta)$-oper} (resp., a {\bf $(\mr{GO}_{2\ell}^0, \vartheta)$-oper}) on $U^\mr{log}/S^\mr{log}$, we mean a pair
\begin{align} \label{Eq220}
\nabla^\diamondsuit_\sphericalangle := (\nabla^\diamondsuit, \omega) \left(\text{resp.,} \  \nabla^\diamondsuit_{\sphericalangle, +} := (\nabla^\diamondsuit_+, \omega_+)  \right)
\end{align}
consisting of an $S^\mr{log}$-connection  $\nabla^\diamondsuit$ on $\mcF_\varTheta$ (resp., $\nabla_+^\diamondsuit$ on  $\mcF_{\varTheta, +}$ with $\nabla^\diamondsuit_+ (\mcF_\varTheta) \subseteq \Omega \otimes \mcF_\varTheta$) and a nondegenerate symmetric $\mcO_U$-bilinear map
$\omega : \mcF_{\varTheta}^{\otimes 2} \rightarrow \mcN_\vartheta$ (resp.,
 $\omega_+ : \mcF_{\varTheta, +}^{\otimes 2} \rightarrow \mcN_\vartheta$)
   such that the collection
\begin{align} \label{Eq221}
\nabla_\sphericalangle^{\diamondsuit \Rightarrow \heartsuit} := (\mcF_{\varTheta}, \nabla^\diamondsuit, \{ \mcF_{\varTheta}^j \}_{j=0}^{2 \ell-1}, \mcN_\varTheta, \nabla_{\mcN_\varTheta}, \omega) \hspace{17mm}\\
\left(\text{resp.,}\nabla_{\sphericalangle, +}^{\diamondsuit \Rightarrow \heartsuit} := (\mcF_{\varTheta, +}, \nabla_+^\diamondsuit, \{ \mcF_{\varTheta, +}^j \}_{j=0}^{2 \ell-1}, \mcN_\varTheta, \nabla_{\mcN_\varTheta}, \omega_+, \varpi_\varTheta) \right) \notag
\end{align}
forms a $\mr{GO}_{2 \ell -1}$-oper (resp., a $\mr{GO}_{2 \ell}^0$-oper) on $U^\mr{log}/S^\mr{log}$.
If $U^\mr{log}/S^\mr{log} = X^\mr{log}/S^\mr{log}$ for some pointed stable curve $\msX := (X/S, \{ \sigma_i \}_i)$,
then any $(\mr{GO}_{2\ell -1}, \vartheta)$-oper (resp., $(\mr{GO}_{2\ell}^0, \vartheta)$-oper) on that log curve will be referred to as a {\bf $(\mr{GO}_{2\ell -1}, \vartheta)$-oper on $\msX$} (resp., a {\bf $(\mr{GO}_{2\ell}^0, \vartheta)$-oper on $\msX$}).
\item[(ii)]
Let $\nabla^\diamondsuit_\sphericalangle$ and  $\nabla'^\diamondsuit_\sphericalangle$
(resp., $\nabla^\diamondsuit_{\sphericalangle, +}$ and  $\nabla'^\diamondsuit_{\sphericalangle, +}$)
 be  $(\mr{GO}_{2 \ell-1}, \vartheta)$-opers   (resp., $(\mr{GO}_{2 \ell}^0, \vartheta)$-opers) on $U^\mr{log}/S^\mr{log}$.
An {\bf isomorphism of   $(\mr{GO}_{2 \ell-1}, \vartheta)$-opers} (resp., {\bf $(\mr{GO}_{2 \ell}^0, \vartheta)$-opers}) from  $\nabla^\diamondsuit_\sphericalangle$ to   $\nabla'^\diamondsuit_\sphericalangle$
(resp., from   $\nabla^\diamondsuit_{\sphericalangle, +}$ to   $\nabla'^\diamondsuit_{\sphericalangle, +}$)
 is defined as an isomorphism   of  $\mr{GO}_{2 \ell-1}$-opers $\nabla_\sphericalangle^{\diamondsuit \Rightarrow \heartsuit} \xrightarrow{\sim} \nabla_\sphericalangle'^{\diamondsuit \Rightarrow \heartsuit}$ (resp., an isomorphism of $\mr{GO}_{2\ell}^0$-opers $\nabla_{\sphericalangle, +}^{\diamondsuit \Rightarrow \heartsuit} \xrightarrow{\sim} \nabla_{\sphericalangle, +}'^{\diamondsuit \Rightarrow \heartsuit}$).
 \end{itemize}
\ede
\SSP

\bpr \label{Rem9}
Each $\mr{GO}_{2\ell}^0$-oper is isomorphic to (the $\mr{GO}_{2\ell}^0$-oper induced by) a $(\mr{GO}_{2\ell}^0, \vartheta')$-oper for some $(2\ell  -1)$-theta characteristic $\vartheta'$.
\epr
\begin{proof}
Let us take a $\mr{GO}_{2\ell}^0$-oper  
$\msF^\heartsuit_{\sphericalangle, +} := (\mcF_+, \nabla_+, \{ \mcF^j_+ \}_{j=0}^{2\ell -1}, \mcN, \nabla_\mcN, \omega_+, \varpi)$ on $U^\mr{log}/S^\mr{log}$.
This $\mr{GO}_{2\ell}^0$-oper determines  a flat vector  bundle
  $(\mcF, \nabla)$ defined as in \S\,\ref{SS23}.
If we set   $\varTheta' := \mcF^{2\ell -2}_+$,
then the composite
\begin{align}
\mcF_{\varTheta'} \left(=\mcD^{<2\ell -1} \otimes \varTheta' \right) \hookrightarrow \mcD^{< \infty} \otimes \mcF \rightarrow \mcF
\end{align}
turns out to be  an isomorphism, 
where the first arrow arises from the natural inclusions $\mcD^{< 2\ell -1}\hookrightarrow \mcD^{< \infty}$ and $\varTheta' \hookrightarrow \mcF$, and the second arrow is the $\mcD^{< \infty}$-action on $\mcF$ corresponding to $\nabla$.
 Also, \eqref{Eq1255}, \eqref{Eq91}, and \eqref{Eq130} in our situation here   give rise to   isomorphisms $\mcN_{\varTheta'} \xrightarrow{\sim} \mcN$, 
 $\overline{\mcF}_{\varTheta'} \xrightarrow{\sim} \overline{\mcF}$.
 In particular,  we obtain  an isomorphism $\eta_{\mcF_+} : (\mcF_{\varTheta'} \oplus \overline{\mcF}_{\varTheta'} = ) \mcF_{\varTheta', +} \xrightarrow{\sim} \mcF_+ (\stackrel{\eqref{Eq52}}{=} \mcF \oplus \overline{\mcF})$.
 The $S^\mr{log}$-connection 
 $\nabla_+$ is transposed into an $S^\mr{log}$-connection $\nabla'_+$ on $\mcF_{\varTheta', +}$  via  this isomorphism.
Also, the $S^\mr{log}$-connection $\mr{det}(\nabla)$ on $\mr{det}(\mcF)$ induced from $\nabla$ corresponds to an $S^\mr{log}$-connection $\nabla_{\vartheta'}$ on $\mcT^{\otimes (2\ell -1)(\ell -1)} \otimes \varTheta'^{\otimes (2\ell -1)}$ via \eqref{Eq98}.
Thus, we obtain an $(2\ell -1)$-theta characteristic  $\vartheta' := (\varTheta', \nabla_{\vartheta'})$ of $U^\mr{log}/S^\mr{log}$.
Moreover, 
if $\omega'_+$ denotes the bilinear map $\mcF_{\varTheta'}^{\otimes 2} \rightarrow \mcN_{\varTheta'}$ corresponding to $\omega_+$ via $\eta_{\mcF_+}$,
then
the resulting pair $\nabla_{\sphericalangle, +}^\diamondsuit := (\nabla'_+, \omega'_+)$   specifies  a $(\mr{GO}_{2\ell}, \vartheta')$-oper with $\nabla_{\sphericalangle, +}^{\diamondsuit \Rightarrow \heartsuit} \cong \msF^\heartsuit_{\sphericalangle, +}$.
This completes the proof of the assertion.
\end{proof}
\SSP

\begin{rem}
[Change of $(2\ell -1)$-theta characteristics]
 \label{Rem99}
Recall  from ~\cite[\S\,4.6.5]{Wak8} that for a flat line bundle $\msL := (\mcL, \nabla_\mcL)$ on $U^\mr{log}/S^\mr{log}$, the pair 
\begin{align}
\vartheta \otimes \msL := (\varTheta \otimes \mcL, \nabla_\vartheta \otimes \nabla_\mcL^{\otimes 2\ell -1})
\end{align}
forms a $(2\ell -1)$-theta characteristic of $U^\mr{log}/S^\mr{log}$.
Conversely, if $\vartheta'$ is another $(2\ell -1)$-theta characteristic,
then there exists a flat bundle $\vartheta' /\vartheta$  such that $\vartheta \otimes (\vartheta'/\vartheta)$ is isomorphic to $\vartheta'$ (cf. ~\cite[Lemma 4.35]{Wak8}).

Now, let us take a $(\mr{GL}_{2\ell}, \vartheta)$-oper $\nabla^\diamondsuit_{\sphericalangle, +} := (\nabla^\diamondsuit_+, \omega_+)$ on $U^\mr{log}/S^\mr{log}$ and $\msL := (\mcL, \nabla_\mcL)$  a flat line bundle on $U^\mr{log}/S^\mr{log}$.
We shall denote by 
$\nabla_{+, \otimes \msL}^\diamondsuit$ 
the $S^\mr{log}$-connection on $\mcF_{\varTheta \otimes \mcL, +}$ corresponding to $\nabla_\mcL \otimes \nabla_{+}^\diamondsuit$ via the  isomorphism
$\mcF_{\varTheta \otimes \mcL, +} \xrightarrow{\sim} \mcL \otimes \mcF_{\varTheta, +}$
defined as the direct sum of $\gamma : \mcF_{\varTheta \otimes \mcL}\xrightarrow{\sim} \mcL \otimes \mcF_{\varTheta}$ constructed in ~\cite[Eq, (586)]{Wak8} and the isomorphism 
$\mcT^{\otimes (\ell -1)} \otimes \varTheta \otimes \mcL \xrightarrow{\sim} \mcL \otimes \mcT^{\otimes (\ell -1)} \otimes \varTheta$ given by $a \otimes b \otimes c \mapsto b \otimes c \otimes a$.
The tensor product of $\omega_+$ and  the identity morphism $\mr{id}_{\mcL^{\otimes 2}}$ specifies  a bilinear map $\omega_{+, \otimes \msL} : \mcF_{\varTheta \otimes \mcL}^{\otimes 2} \rightarrow \mcN \otimes \mcL^{\otimes 2}$ under the identification $\mcF_{\varTheta \otimes \mcL} = \mcL \otimes \mcF_\varTheta$ given by $\gamma$.
Then, the resulting pair 
\begin{align} \label{Eq1404}
\nabla^\diamondsuit_{\sphericalangle, +, \otimes \msL}
 := (\nabla_{\sphericalangle, \otimes \msL}^\diamondsuit, \omega_{+, \otimes \msL})
\end{align}
 forms a $(\mr{GO}_{2 \ell}, \vartheta \otimes \msL)$-oper on $U^\mr{log}/S^\mr{log}$, which may be thought of as a {\it twist} of $\nabla^\diamondsuit_{\sphericalangle, +}$ by $\msL$.
\end{rem}

\LSP
\subsection{The relation between $(\mr{GO}_{2 \ell}^0, \vartheta)$-opers and $(\mr{GO}_{2 \ell -1}, \vartheta)$-opers} \label{SS670}

Note that  $(\mr{GO}_{2\ell -1}, \vartheta)$-opers (resp., $(\mr{GO}_{2\ell}^0, \vartheta)$-opers) may be considered as  $\mr{GO}_{2\ell -1}$-opers (resp., $\mr{GO}_{2 \ell}^0$-opers) via the functor $(-)^{\Rightarrow \heartsuit}$,
so they  form a full subcategory
\begin{align} \label{Eq92}
\mcO p^{+}_{2\ell -1, \vartheta} (U^\mr{log}) \ \left(\text{resp.,} \mcO p_{2\ell, \vartheta} (U^\mr{log}) \right)
\end{align}
of $\mcO p^{+}_{2\ell -1} (U^\mr{log})$ (resp., $\mcO p_{2\ell} (U^\mr{log})$).

\SSP
\begin{prdef}\label{Prop26}
Let us keep the above notation.
Then, \eqref{Eq68} restricts to an equivalence of categories
\begin{align} \label{Eq138}
\mcO p_{2 \ell, \vartheta} (U^\mr{log}) \xrightarrow{\sim} \mcO p_{2 \ell -1, \vartheta}^+ (U^\mr{log}).
\end{align}
Moreover, the formation of this equivalence commutes with pull-back  over any \'{e}tale $U$-scheme  (in the same sense as \eqref{Eq68}), 
 as well as with base-change over any  fs log scheme over $S^\mr{log}$.

 For a $(\mr{GO}_{2\ell}^0, \vartheta)$-oper $\nabla_{\sphericalangle, +}^\diamondsuit$, the $(\mr{GO}_{2\ell -1}, \vartheta)$-oper and the
 morphism $\overline{\mcF}_\varTheta \rightarrow \mcF_\varTheta$
   associated to $\nabla_{\sphericalangle, +}^\diamondsuit$ via \eqref{Eq138} will be denoted by $\nabla_{\sphericalangle, + \Rightarrow \emptyset}^\diamondsuit$ and $\nu (\nabla_{\sphericalangle, +}^\diamondsuit )$, respectively.
\end{prdef}
\begin{proof}
The assertion follows from the various definitions involved (including the construction of the equivalence \eqref{Eq68}).
\end{proof}
\SSP

Denote by $\vartheta_0$ the  $(2\ell -1)$-theta characteristic  $(\varTheta_0, d)$ of $U^\mr{log}/S^\mr{log}$, where $\varTheta_0 := \Omega^{\otimes (\ell-1)}$ and  we regard the universal derivation $d : \mcO_U \rightarrow \Omega$
as an $S^\mr{log}$-connection on $\mcT^{\otimes   (2\ell -1)(\ell -1)} \otimes (\Omega^{\otimes (\ell-1)})^{\otimes (2\ell -1)}$
via the  identification
$\mcT^{\otimes   (2\ell -1)(\ell -1)} \otimes (\Omega^{\otimes (\ell-1)})^{\otimes (2\ell -1)} = \mcO_U$ induced by  $\mcT \otimes \Omega = \mcO_U$.

Now, let us take a $(\mr{GO}_{2\ell -1}, \vartheta)$-oper
$\nabla^\diamondsuit_{\sphericalangle} := (\nabla^\diamondsuit,  \omega)$ on $U^\mr{log}/S^\mr{log}$.
Consider the composite
\begin{align} \label{Eq193}
\mcF_{\varTheta_0} \left(= \mcD^{< (2\ell -1)} \otimes \varTheta_0 \right)
\hookrightarrow \mcD^{< \infty} \otimes (\overline{\mcF}_\varTheta^\vee \otimes \mcF_\varTheta) \rightarrow  \overline{\mcF}_\varTheta^\vee \otimes \mcF_\varTheta,
\end{align}
where the first arrow arises from the inclusions $\mcD^{< (2\ell -1)} \hookrightarrow \mcD^{< \infty}$ and $\varTheta_0 \left(= \overline{\mcF}_\varTheta^\vee \otimes \mcF_\varTheta^{2\ell -2} \right) \hookrightarrow \overline{\mcF}_\varTheta^\vee \otimes \mcF_\varTheta$, and  the second arrow denotes the $\mcD^{< \infty}$-action on $\overline{\mcF}_\varTheta^\vee \otimes \mcF_\varTheta$ determined by $\nabla^\vee_{\overline{\mcF}_\varTheta} \otimes \nabla^\diamondsuit$.
Since $(\mcF_\varTheta, \nabla^\diamondsuit, \{ \mcF^j_\varTheta \}_j)$ forms a $\mr{GL}_{2\ell -1}$-oper,
this composite turns out to be  an isomorphism.
By means of this isomorphism, $\nabla^\diamondsuit$ may be transposed into an $S^\mr{log}$-connection
$\nabla^\diamondsuit_{0}$ on $\mcF_{\varTheta_0}$. 
The tensor product of $\omega$ and the canonical isomorphism $(\overline{\mcF}^\vee)^{\otimes 2} \xrightarrow{\sim} \mcN_{\varTheta}$ induces, via  \eqref{Eq193}, an $\mcO_U$-valued  bilinear map $\omega_0 : \mcF_{\varTheta_0}^{\otimes 2} \rightarrow \mcO_U$ on $\mcF_{\varTheta_0}$.
It is verified that 
the resulting pair
\begin{align} \label{Eq195}
\nabla^\diamondsuit_{\sphericalangle, 0} := (\nabla_0^\diamondsuit, \omega_0)
\end{align}
 forms a $(\mr{GO}_{2\ell -1}, \vartheta_0)$-oper on $U^\mr{log}/S^\mr{log}$.
Hence,  for each $(\mr{GO}_{2\ell}^0, \vartheta)$-oper
$\nabla^\diamondsuit_{\sphericalangle, +} := (\nabla^\diamondsuit_+,  \omega_+)$  with $\nabla^\diamondsuit_{\sphericalangle, +\Rightarrow \emptyset} = \nabla^\diamondsuit_{\sphericalangle}$ (cf. Proposition-Definition \ref{Prop26}),
the associated morphism 
$\nu (\nabla^\diamondsuit_{\sphericalangle, +})$ may be regarded as   an element of $H^0 (X, \mcF_{\varTheta_0})$ via  \eqref{Eq193}.

\LSP
\subsection{The case of pointed stable curves} \label{SS671}

In the case where the underlying log curve arises from a pointed stable curve,
the above proposition implies the following assertion.

\SSP
\bpr \label{Prop155}
Suppose that $U^\mr{log}/S^\mr{log} = X^\mr{log}/S^\mr{log}$ for some pointed stable curve $\msX := (X/S, \{ \sigma_i \}_{i=1}^r)$  over an affine $k$-scheme $S$.
Then, there exists a canonical bijection of sets
\begin{align} \label{Eq140}
\left\{ \begin{matrix}  \text{isomorphism classes of} \\ \text{$(\mr{GO}_{2\ell}^0, \vartheta)$-opers on $\msX$} \end{matrix}\right\}
\xrightarrow{\sim}
\left\{ \begin{matrix}  \text{isomorphism classes of} \\ \text{$(\mr{GO}_{2 \ell -1}, \vartheta)$-opers on $\msX$} \end{matrix}\right\} \times H^0 (X, \Omega^{\otimes \ell}).
\end{align}
Moreover,  the  formation of this bijection   commutes with base-change to $S$-schemes.
\epr
\begin{proof}
Let us take a $(\mr{GO}_{2 \ell -1},\vartheta)$-oper  $\nabla_\sphericalangle^\diamondsuit := (\nabla^\diamondsuit, \omega)$ on $\msX$.
For simplicity, we write 
$\nabla := \nabla_0^\diamondsuit$.
By Proposition-Definition \ref{Prop26} and the discussion following that proposition,   the set of isomorphism classes of 
$(\mr{GO}^0_{2\ell}, \vartheta)$-opers $\nabla_{\sphericalangle, +}^\diamondsuit$ on $\msX$ with  $\nabla_{\sphericalangle, +\Rightarrow \emptyset}^\diamondsuit = \nabla_{\sphericalangle}^\diamondsuit$
 is  in bijection with (the underlying set of) the  cokernel
of the morphism
\begin{align}
H^0 (\nabla) : H^0 (X, \mcF_{\varTheta_0}) \rightarrow H^0 (X, \Omega \otimes \mcF_{\varTheta_0})
\end{align}
induced by $\nabla$.
Thus, the assertion follows from Lemma \ref{Lem19} below.
\end{proof}

\SSP
\ble \label{Lem19}
Let us keep the notation in the proof of Proposition \ref{Prop155}.
Then,  the $k$-linear composite 
\begin{align} \label{Eq177}
H^0 (X, \Omega^{\otimes \ell}) \xrightarrow{} H^0 (X, \Omega \otimes \mcF_{\varTheta_0}) \twoheadrightarrow \mr{Coker}(H^0 (\nabla))
\end{align}
is bijective, 
where 
the first arrow arises from the inclusion $\Omega^{\otimes \ell} \left(= \Omega \otimes \mcF_{\varTheta_0}^{2\ell -2} \right) \hookrightarrow \Omega \otimes \mcF_{\varTheta_0}$.
\ele
\begin{proof}
For each $j =1, \cdots, 2\ell -1$, we shall write
\begin{align}
\nabla^j : \mcF_{\varTheta_0}^j \rightarrow \Omega \otimes \mcF_{\varTheta_0}^{j-1}
\end{align}
for the morphism obtained by restricting $\nabla$ and write $H^0 (\nabla^j) : H^0 (X, \mcF_{\varTheta_0}^j) \rightarrow H^0 (X, \Omega \otimes \mcF_{\varTheta_0}^{j-1})$ for the associated   morphism of $k$-vector spaces.
Since  $\mcF_{\varTheta_0}^j / \mcF_{\varTheta_0}^{j+1} \xrightarrow{\sim} \mcT^{\otimes (\ell -1-j)}$,
the morphism $H^0 (X, \mcF_{\varTheta_0}^{\ell -1}) \rightarrow H^0 (X, \mcF_{\varTheta_0})$ and $H^0 (X, \Omega \otimes \mcF_{\varTheta_0}^{\ell -2}) \rightarrow H^0 (X, \Omega \otimes \mcF_{\varTheta_0})$ induced from the natural inclusions $\mcF_{\varTheta_0}^{\ell -1} \hookrightarrow \mcF_{\varTheta_0}$ and $\Omega \otimes \mcF_{\varTheta_0}^{\ell -2} \hookrightarrow \Omega\otimes \mcF_{\varTheta_0}$, respectively, are bijective.
This implies that the natural morphism $\mr{Coker}(H^0 (\nabla^{\ell -1})) \rightarrow \mr{Coker}(H^0 (\nabla))$ is bijective.

Next, for each $j= \ell -1, \cdots, 2\ell -2$,  consider the following morphism of short exact sequences:
\begin{align} \label{Eq49}
\vcenter{\xymatrix@C=46pt@R=36pt{
0\ar[r] & \mcF_{\varTheta_0}^{j+1} \ar[r]^-{\mr{inclusion}}  \ar[d]^-{\nabla^{j+1}} &  \mcF_{\varTheta_0}^{j} \ar[r]^-{\mr{quotient}} \ar[d]^-{\nabla^j} &  \mcF_{\varTheta_0}^j /\mcF_{\varTheta_0}^{j+1} \ar[r] \ar[d] & 0 
\\
0 \ar[r] & \Omega \otimes \mcF_{\varTheta_0}^{j} \ar[r]_-{\mr{inclusion}} & \Omega \otimes \mcF_{\varTheta_0}^{j-1}  \ar[r]_-{\mr{quotient}} & \Omega \otimes (\mcF_{\varTheta_0}^{j-1}/\mcF_{\varTheta_0}^j) \ar[r] & 0.
 }}
\end{align}
The right-hand vertical arrow is an isomorphism because  $(\mcF_{\varTheta_0}, \nabla, \{ \mcF_{\varTheta_0}^j \}_{j})$ forms a $\mr{GL}_{2\ell -1}$-oper.
Hence, this diagram induces a morphism of exact sequences of $k$-vector spaces
\begin{align} \label{Eq49}
\vcenter{\xymatrix@C=10pt@R=36pt{
0\ar[r] &H^0 (X, \mcF_{\varTheta_0}^{j+1})  \ar[r]  \ar[d]^-{H^0(\nabla^{j+1})} &  H^0 (X, \mcF_{\varTheta_0}^{j}) \ar[r] \ar[d]^-{H^0(\nabla^j)} &  H^0 (X, \mcF_{\varTheta_0}^j /\mcF_{\varTheta_0}^{j+1}) \ar[r] \ar[d]^-{\wr} & H^1 (X, \mcF_{\varTheta_0}^{j+1}) \ar[d]^-{H^1 (\nabla^{j+1})} 
\\
0 \ar[r] & H^0 (X, \Omega \otimes \mcF_{\varTheta_0}^{j}) \ar[r] & H^0 (X, \Omega \otimes \mcF_{\varTheta_0}^{j-1})  \ar[r] & H^0 (X, \Omega \otimes (\mcF_{\varTheta_0}^{j-1}/\mcF_{\varTheta_0}^j)) \ar[r] & H^1 (X, \Omega \otimes \mcF_{\varTheta_0}^{j}).
 }}
\end{align}
If  $j =\ell -1$,
then the equalities  $h^0 (\mcF_{\varTheta_0}^{j}/\mcF_{\varTheta_0}^{j+1}) = h^0 (\Omega \otimes (\mcF_{\varTheta_0}^{j-1}/\mcF_{\varTheta_0}^j))=1$ hold, and hence, 
the natural morphism $\mr{Coker} (H^0 (\nabla^{\ell})) \rightarrow \mr{Coker}(H^0 (\nabla^{\ell -1}))$ is bijective.
On the other hand,
if $j > \ell -1$,
then  the equalities  $h^1 (\mcF_{\varTheta_0}^{j+1}) = h^1 (\Omega \otimes \mcF_{\varTheta_0}^j) = 0$, so
the snake lemma applied to this diagram shows that
the morphism 
$\mr{Coker} (H^0 (\nabla^{j+1})) \rightarrow \mr{Coker}(H^0 (\nabla^{j}))$ is bijective.
By  the observations made so far, 
the morphism 
\begin{align}
H^ 0(X, \Omega^{\otimes \ell}) \left(= \mr{Coker}(H^0 (\nabla^{2\ell -1})) \right)\rightarrow \mr{Coker}(H^0 (\nabla))
\end{align}
turns out to be  bijective.
This completes the proof of this assertion.
\end{proof}

\LSP
\subsection{The relation between $\mfs \mfo_{2\ell}$-opers and $\mfs \mfo_{2\ell -1}$-opers
} \label{SS29}

Let   us take a  $\mr{GO}^0_{2\ell}$-oper   $\msF^\heartsuit_{\sphericalangle, +} := (\mcF, \nabla, \{ \mcF^j \}_j, \mcN, \nabla_\mcN, \omega, \varpi)$ on $U^\mr{log}/S^\mr{log}$.
Then, $(\mcF, \{ \mcF^j \}_{j=0}^{2 \ell}, \mcN, \omega, \varpi)$
induces a $B$-bundle $\mcE_B$ on $U$ via projectivization, i.e., via change of structure group by the projection $\mr{GO}_{2\ell}^0 \twoheadrightarrow \mr{PGO}^0_{2\ell}$.
Moreover, $(\nabla, \nabla_\mcN)$ determines an $S^\mr{log}$-connection  $\nabla_\mcE$ on the $\mr{PGO}_{2 \ell}^0$-bundle  $\mcE := \mcE_B \times^B \mr{PGO}_{2 \ell}^0$.
Just as in the discussion of  ~\cite[(c).\,2.9]{BeDr2},
the resulting pair 
\begin{align} \label{Eq37}
\msF^{\heartsuit \Rightarrow \spadesuit}_{\sphericalangle, +} :=(\mcE_B, \nabla_\mcE)
\end{align}
 specifies an $\mfs \mfo_{2 \ell}$-oper.
We here omit 
 the precise definition of a $\mfg$-oper for
a simple Lie algebra $\mfg$.
For its details (in the case where the underlying curve is a pointed stable curve), we refer the reader  to ~\cite[Definition 2.1]{Wak8}.

\SSP
\bpr \label{Prop28}
Assume  that $H^2 (U, \mcT^{\otimes m}) = 0$  for every integer $m$ with  $ -\ell +1 \leq m \leq  \ell -1$.
Then, the assignment  $(\nabla_\sphericalangle^{\diamondsuit \Rightarrow \heartsuit})^{\Rightarrow \spadesuit}$ determines a bijection of sets
\begin{align} \label{Eq76}
\left\{ \begin{matrix}  \text{isomorphism classes of} \\ \text{$(\mr{GO}_{2\ell}^0, \vartheta)$-opers on $U^\mr{log}/S^\mr{log}$} \end{matrix}\right\}
\xrightarrow{\sim}
\left\{ \begin{matrix}  \text{isomorphism classes of} \\ \text{$\mfs \mfo_{2\ell}$-opers on $U^\mr{log}/S^\mr{log}$} \end{matrix}\right\}.
\end{align}
(Note that the assumption imposed above is fulfilled when $U^\mr{log}/S^\mr{log} = X^\mr{log}/S^\mr{log}$ for some pointed stable curve $\msX := (X/S, \{ \sigma_i \}_{i=1}^r)$  over an affine $k$-scheme $S$.
In that case, the formation of the bijection \eqref{Eq76} commutes with 
 base-change to affine schemes  over $S$.)
\epr
\begin{proof}
Since
the algebraic group $\mr{GO}_{2\ell}^0$ admits a natural inclusion $\mr{GO}_{2\ell}^0 \hookrightarrow \mr{GL}_{2\ell}$,
 the injectivity of the map \eqref{Eq76} follows from ~\cite[Proposition 4.22, (ii)]{Wak8}.

 Next, 
we shall  consider the  surjectivity of \eqref{Eq76}.
Let $\mcE^\spadesuit_+$ be an $\mfs \mfo_{2\ell}$-oper on $U^\mr{log}/S^\mr{log}$.
There exists a  covering $\{ U_\alpha \}_{\alpha \in I}$ (where $I$ denotes an index set) of $U$ in the \'{e}tale topology 
such that the restriction  $\msE^\spadesuit_+ |_{U_\alpha}$ to each $U_\alpha$ arises, via projection,  from
a $(\mr{GO}_{2\ell}^0, \vartheta_\alpha)$-oper $\nabla^\diamondsuit_{\sphericalangle, +, \alpha}$ on $U_\alpha^\mr{log}/S^\mr{log}$ for some $(2\ell -1)$-theta characteristic $\vartheta_\alpha$ of $U_\alpha^\mr{log}/S^\mr{log}$ (cf. Proposition \ref{Rem9}).
After possibly  tensoring  $\nabla^\diamondsuit_{\sphericalangle, +, \alpha}$ with $\vartheta / \vartheta_\alpha$ (cf. Remark \ref{Rem99}),
we may assume that 
 $\nabla^\diamondsuit_{\sphericalangle, +, \alpha}$
 is a $(\mr{GO}_{2\ell}^0, \vartheta)$-oper.
By Proposition-Definition  \ref{Prop26},
 $\nabla^\diamondsuit_{\sphericalangle, +, \alpha}$
 corresponds to a pair $(\nabla^\diamondsuit_{\sphericalangle,  \alpha}, \nu_\alpha)$, where  $\nabla^\diamondsuit_{\sphericalangle,  \alpha}$ denotes a $(\mr{GO}_{2\ell -1}, \vartheta)$-oper on $U^\mr{log}_\alpha/S^\mr{log}$.
If $U_{\alpha, \beta} := U_\alpha \cap U_\beta \neq \emptyset$ ($\alpha, \beta \in I$),
then since  the restrictions  
  $\nabla^\diamondsuit_{\sphericalangle,  \alpha} |_{U_{\alpha, \beta}}$ and   $\nabla^\diamondsuit_{\sphericalangle,  \beta}  |_{U_{\alpha, \beta}}$  are isomorphic  to each other via  taking their projectivizations, it follows from ~\cite[Proposition 4.22, (ii)]{Wak8} again 
  that $\nabla^\diamondsuit_{\sphericalangle,  \alpha} |_{U_{\alpha, \beta}} \cong \nabla^\diamondsuit_{\sphericalangle,  \beta} |_{U_{\alpha, \beta}}$.
  Then, it follows from ~\cite[Proposition 5.6]{Wak8} that $\nabla^\diamondsuit_{\sphericalangle,  \alpha}$ may be glued together to obtain a $(\mr{GO}_{2\ell -1}, \vartheta)$-oper 
  $\nabla^\diamondsuit_{\sphericalangle}$ on $U^\mr{log}/S^\mr{log}$.
 Moreover,  let us  replace $\{ U_\alpha \}_\alpha$ with its refinement, and  suppose  that, for  any pair $(\alpha, \beta) \in I^2$ with  $U_{\alpha, \beta} \neq \emptyset$,
  there exists a section $\mu_{\alpha, \beta} \in H^0 (U_{\alpha, \beta}, \overline{\mcF}_\varTheta^\vee \otimes \mcF_\varTheta)$ with $\nu_\alpha - \nu_\beta = (\nabla^\vee_{\overline{\mcF}_\varTheta} \otimes \nabla^\diamondsuit) (\mu_{\alpha, \beta})$.
  Here, recall   that $\mcF_{\varTheta_0} := \overline{\mcF}_\varTheta^\vee \otimes \mcF_\varTheta$ admits a filtration
    whose subquotients are of the form $\mcT^{\otimes m}$ (with $- \ell + 1 \leq m \leq \ell -1$).
  By assumption, we  have  
  $H^2 (U, \overline{\mcF}_\varTheta^\vee \otimes \mcF_\varTheta) = 0$.
  This implies that 
  the collection  $(\mu_{\alpha, \beta})_{\alpha, \beta}$ form a \v{C}ech $1$-cocycle after possibly replacing $\mu_{\alpha, \beta}$ with another.
  By means of the automorphisms  $\mr{id} + \mu_{\alpha, \beta}$ of $\mcF_\varTheta |_{U_{\alpha, \beta}}$  for various $(\alpha, \beta)$'s,
   the  $\nabla^\diamondsuit_{\sphericalangle, +, \alpha}$'s may be glued together  to
  obtain a $(\mr{GO}_{2\ell}^0, \vartheta)$-oper  $\nabla^\diamondsuit_{\sphericalangle, +}$ on $U^\mr{log}/S^\mr{log}$.
  Since $\msE^\spadesuit_+$ does not admit  nontrivial automorphisms (cf. ~\cite[Proposition 2.9]{Wak8}), we have
   $(\nabla^{\diamondsuit \Rightarrow \heartsuit}_{\sphericalangle, +})^{\Rightarrow \spadesuit} \cong \msE^\spadesuit_+$.
  This completes the proof of the surjectivity of  \eqref{Eq76}.
\end{proof}
\SSP

By combining Propositions \ref{Prop155}, \ref{Prop28} and ~\cite[Theorem 5.12]{Wak8}, we obtain the following assertion.

\SSP
\begin{tdef} \label{Cor4}
Suppose that $U^\mr{log}/S^\mr{log} = X^\mr{log}/S^\mr{log}$ for some pointed stable curve $\msX := (X/S, \{ \sigma_i \}_{i=1}^r)$ over an affine $k$-scheme $S$.
Then, there exists a canonical bijection of sets
\begin{align} \label{Eq118}
\left\{ \begin{matrix}  \text{isomorphism classes of} \\ \text{$\mfs \mfo_{2 \ell}$-opers on $\msX$} \end{matrix}\right\}
\xrightarrow{\sim}
\left\{ \begin{matrix}  \text{isomorphism classes of} \\ \text{$\mfs \mfo_{2 \ell -1}$-opers on $\msX$} \end{matrix}\right\} \times H^0 (X, \Omega^{\otimes \ell}).
\end{align}
Moreover, the formation of this bijection commutes with base-change to affine schemes over $S$.

If we are given  an $\mfs \mfo_{2\ell}$-oper $\msE^\spadesuit_{+}$, 
then
 the $\mfs \mfo_{2\ell -1}$-oper  and the element of $H^0 (X, \Omega^{\otimes \ell})$ associated to $\msE^\spadesuit_+$ via \eqref{Eq118} will be denoted by 
 $\msE^\spadesuit_{+ \Rightarrow \emptyset}$
 and $\nu (\msE^\spadesuit_+)$, respectively.
\end{tdef}

\LSP
\subsection{The moduli stack of  $\mfs \mfo_{2\ell}$-opers} \label{SS30}

Denote by $\mcS ch_{/k}$ the category of $k$-schemes.
For $s \in \{ 2\ell -1, 2 \ell \}$,
we shall write
 \begin{align}
 \mcO p_{s, g, r}
 \end{align}
 for the category over $\mcS ch_{/k}$ defined as follows:
 \begin{itemize}
 \item
 The objects are  pairs $(\msX, \msE^\spadesuit)$, where $\msX$ denotes an $r$-pointed stable curve of genus $g$ over a $k$-scheme $S$ and $\msE^\spadesuit$ denotes an $\mfs \mfo_s$-oper  on $\msX$;
 \item
 The morphisms from $(\msX, \msE^\spadesuit)$ to $(\msX', \msE'^\spadesuit)$ are morphisms of $r$-pointed curves $(\phi, \Phi) : \msX \rightarrow \msX'$, in the sense of   ~\cite[Definition 1.36, (ii)]{Wak8}, satisfying  $\msE^\spadesuit \cong \phi^* (\msE'^\spadesuit)$;
 \item
 The projection $\mcO p_{s, g,r} \rightarrow \mcS ch_{/k}$ is given by assigning, to each pair $(\msX, \msE^\spadesuit)$ as above, the base scheme $S$ of $\msX$. 
 \end{itemize}

The assignment $(\msX, \msE^\spadesuit) \mapsto \msX$ defines a morphism from $\mcO p_{s, g, r}$ to the moduli stack
$\overline{\mcM}_{g, r}$.
According to ~\cite[Theorem A]{Wak8}, $\mcO p_{s, g, r}$ may be represented by a smooth Deligne-Mumford stack over $k$ and forms a relative affine scheme over $\overline{\mcM}_{g, r}$.
The assignment $(\msX, \msE^\spadesuit_+) \mapsto (\msX, \msE^\spadesuit_{+ \Rightarrow \emptyset})$ (cf. Theorem-Definition  \ref{Cor4}) determines a morphism  
\begin{align} \label{Eq125}
\mcO p_{2\ell, g, r} \rightarrow 
\mcO p_{2\ell -1, g, r}
\end{align}
over $\overline{\mcM}_{g, r}$, by which we regard $\mcO p_{2\ell, g, r}$  as a stack over $\mcO p_{2\ell -1, g, r}$.

Next, we shall write $\Omega_{\mr{univ}} := \Omega_{\mcC_{g, r}^\mr{log}/\overline{\mcM}_{g, r}^\mr{log}}$, and write
\begin{align}
\mcU := \mcO p_{2 \ell -1, g, r} \times_{\overline{\mcM}_{g, r}} \mbV (f_{\mr{univ}*}(\Omega^{\otimes \ell}_{\mr{univ}})),
\end{align}
where 
$ \mbV (f_{\mr{univ}*}(\Omega^{\otimes \ell }_{\mr{univ}}))$ denotes the relative affine scheme determined by  the vector bundle 
$f_{\mr{univ}*}(\Omega^{\otimes \ell }_{\mr{univ}})$, i.e., 
the spectrum 
 of the  symmetric algebra $\mbS_{\mcO_{\overline{\mcM}_{g, r}}} (f_{\mr{univ}*}(\Omega^{\otimes \ell }_{\mr{univ}})^\vee)$ 
over $\mcO_{\overline{\mcM}_{g, r}}$ associated to the dual of 
    $f_{\mr{univ}*}(\Omega^{\otimes \ell }_{\mr{univ}})$.
By  Theorem-Definition \ref{Cor4}, the assignment $(\msX, \msE^\spadesuit_+) \mapsto ((\msX, \msE^\spadesuit_{+ \Rightarrow \emptyset}), \nu (\msE_+^\spadesuit))$ defines an isomorphism of stacks
\begin{align} \label{Eq72}
\mcO p_{2\ell, g, r} \xrightarrow{\sim} \mcU
\end{align}
over $\mcO p_{2 \ell -1, g, r}$.
In particular,  the morphism \eqref{Eq125} is surjective.
Moreover, by using  this isomorphism,
we equip  $\mcO p_{2\ell, g, r}$ with a structure of relative affine space on $\mcO p_{2 \ell -1, g, r}$ modeled on $\mbV (f_{\mr{univ}*}(\Omega^{\otimes \ell}_{\mr{univ}}))$.

\vspace{10mm}
\section{The moduli space of dormant $\mfs \mfo_{2 \ell}$-opers} \label{S31}
\LSP

This section focuses on $\mfs \mfo_{2 \ell}$-opers in  characteristic $p> 0$ and discusses those  with   vanishing $p$-curvature, i.e., dormant  $\mfs \mfo_{2 \ell}$-opers.
Considering  a cohomological description of infinitesimal  deformations,
we prove 
the generic \'{e}taleness of the  moduli space of dormant $\mfs \mfo_{2 \ell}$-opers (cf. Theorem \ref{Thm1}), which is the main result of the present paper.
Note that our proof is reduced  to the case where the underlying curve is  a $3$-pointed projective line by degenerating 
the underlying curve and then detaching its irreducible components.
This approach  is based on the proof of the generic \'{e}taleness for dormant $\mfs \mfl_n$-opers given in ~\cite{Wak8}.

In the rest of the present paper, we suppose that 
the characteristic $\mr{char}(k)$ of $k$ coincides with a prime number $p$
 with $p > 2 (2\ell -1)$.

\LSP
\subsection{Dormant $\mfs \mfo_{2\ell}$-opers} \label{SS45}

Let $U^\mr{log}/S^\mr{log}$ be as before
 and $\vartheta := (\varTheta, \nabla_\vartheta)$ a $(2\ell -1)$-theta characteristic of $U^\mr{log}/S^\mr{log}$ such that $\nabla_\vartheta$ has vanishing $p$-curvature.
(For the definition of $p$-curvature in the logarithmic setting, we refer to, e.g.,  ~\cite[Definitions 3.8 and 4.58]{Wak8}.
Also, it follows from the comment in ~\cite[\S\,4.6.4]{Wak8} that there always exists a $(2\ell -1)$-theta characteristic with vanishing $p$-curvature.)
In ~\cite[Definition 3.15]{Wak8},
we  defined the notion of a   dormant $\mfg$-oper, where $\mfg$ is  a  Lie algebra with certain conditions.
Similarly to  that notion,  we make the following definition.

\SSP
\bde \label{Def4}
Let 
 $\nabla_{\sphericalangle}^\diamondsuit := (\nabla^\diamondsuit, \omega)$
(resp., $\nabla_{\sphericalangle, +}^\diamondsuit := (\nabla^\diamondsuit_+, \omega_+)$) be a $(\mr{GO}_{2\ell -1}, \vartheta)$-oper (resp., a $(\mr{GO}_{2 \ell}^0, \vartheta)$-oper).
We say that 
$\nabla_{\sphericalangle}^\diamondsuit$
(resp., $\nabla_{\sphericalangle, +}^\diamondsuit$) is {\bf dormant} if
$\nabla^\diamondsuit$
(resp., $\nabla^\diamondsuit_+$) has vanishing $p$-curvature.
\ede
\SSP

Since $\nabla_\vartheta$ has vanishing $p$-curvature,
a $(\mr{GO}_{2\ell -1}, \vartheta)$-oper
(resp., a $(\mr{GO}_{2\ell}^0, \vartheta)$-oper)
 is dormant  if and only if the  $\mfs \mfo_{2\ell -1}$-oper (resp., $\mfs \mfo_{2\ell}$-oper) obtained from it via projectivization is dormant (cf. ~\cite[Remark 4.59]{Wak8}).
In particular, Proposition \ref{Prop28} induces  the following assertion.

\SSP
\bpr \label{Prop15}
Assume that $H^2 (U, \mcT^{\otimes m}) = 0$ for every integer $m$ with $-\ell +1 \leq m \leq \ell -1$.
Then, 
the bijection \eqref{Eq76} restricts to a bijection of sets
\begin{align} \label{Eq77}
\left\{ \begin{matrix}  \text{isomorphism classes of} \\ \text{dormant $(\mr{GO}_{2\ell}^0, \vartheta)$-opers on $U^\mr{log}/S^\mr{log}$} \end{matrix}\right\}
\xrightarrow{\sim}
\left\{ \begin{matrix}  \text{isomorphism classes of} \\ \text{dormant $\mfs \mfo_{2\ell}$-opers on $U^\mr{log}/S^\mr{log}$} \end{matrix}\right\}.
\end{align}
Moreover,  
the comment in parentheses described  in the statement of Proposition \ref{Prop28} is  also true in this case.
\epr
\SSP

For $s \in \{ 2\ell -1, 2\ell \}$,
we shall set
\begin{align} \label{essf3}
\mcO p_{s, g, r}^{^\mr{Zzz...}}
\end{align}
to be the closed substack of  $\mcO p_{s, g, r}$ classifying dormant $\mfs \mfo_s$-opers; this admits the natural projection
\begin{align} \label{erde}
\Pi_{s, g, r} : \mcO p_{s, g, r}^{^\mr
{Zzz...}} \rightarrow \overline{\mcM}_{g, r}.
\end{align}
(Hence, $\mcO p_{s, g, r}^{^\mr{Zzz...}} = \mcO p_{\mfs \mfo_s, g, r}^{^\mr{Zzz...}}$  and $\Pi_{s, g, r} = \Pi_{\mfs \mfo_s, g, r}$ in the terminology of Introduction.)
It follows from ~\cite[Theorem C]{Wak8} that  $\mcO p_{s, g, r}^{^\mr{Zzz...}}$ 
may be represented by a nonempty proper Deligne-Mumford stack over $k$ of dimension $3g-3+r$, and     $\Pi_{s, g, r}$ 
is surjective and  finite.

If $\msE^\spadesuit_+$ is a dormant $\mfs \mfo_{2\ell}$-oper, then the induced $\mfs \mfo_{2\ell -1}$-oper $\msE^\spadesuit_{+ \Rightarrow \emptyset}$ (cf. Theorem-Definition \ref{Cor4}) is dormant because of the construction of \eqref{Eq118}.
Hence,
the morphism \eqref{Eq125}
 restricts to a morphism  of stacks
\begin{align} \label{Eq120}
\chi : 
\mcO p_{2 \ell, g, r}^{^\mr{Zzz...}} \rightarrow \mcO p_{2 \ell -1, g, r}^{^\mr{Zzz...}}.
\end{align}
The assignment from each dormant 
$\mfs \mfo_{2\ell -1}$-oper
$\msE^\spadesuit$
 to the dormant 
 $\mfs \mfo_{2\ell}$-oper
  corresponding to the pair $(\msE^\spadesuit, 0)$ via  \eqref{Eq118} determines a section $\mcO p_{2 \ell -1, g, r}^{^\mr{Zzz...}} \rightarrow \mcO p_{2 \ell, g, r}^{^\mr{Zzz...}}$ of $\chi$.
  In particular, $\chi$ turns out to be surjective.

\LSP
\subsection{A cohomological criterion for unramifiedness} \label{SS32}

Let $\msX := (X/k, \{ \sigma_i \}_{i=1}^r)$ be  an $r$-pointed
``{\it smooth}" 
 curve  of genus $g$ over $k$ and $\vartheta := (\varTheta, \nabla_\vartheta)$ a $(2\ell -1)$-theta characteristic of $X^\mr{log}/k$ such that $\nabla_\vartheta$ has vanishing $p$-curvature.
Denote by $X^{(1)}$ the Frobenius twist of $X$ over $k$ and  by $F$ the relative Frobenius morphism $X \rightarrow X^{(1)}$ of $X/k$.

We shall  take a  dormant $(\mr{GO}_{2\ell }, \vartheta)$-oper 
$\nabla^\diamondsuit_{\sphericalangle, +} := (\nabla^\diamondsuit_+, \omega_+)$
  on $\msX$.
  Write $\nabla^\diamondsuit_{\sphericalangle, +\Rightarrow \emptyset} := \nabla^\diamondsuit_{\sphericalangle}  = (\nabla^\diamondsuit, \omega)$.
  Also, the associated $(\mr{GO}_{2\ell -1}, \vartheta_0)$-oper, i.e., ``$\nabla^\diamondsuit_{\sphericalangle, 0} $" constructed as in \eqref{Eq195},  will be denoted by $(\nabla, \omega_0)$.
  
  \SSP
\bpr \label{Prop1001}
Denote by $q$ the $k$-rational point of $\mcO p_{2\ell, g, r}^{^\mr{Zzz...}}$
classifying the dormant $\mfs \mfo_{2\ell }$-oper $\msE^\spadesuit_+ := (\nabla^{\diamondsuit \Rightarrow \heartsuit}_{\sphericalangle, +})^{\Rightarrow \spadesuit}$ (cf. \eqref{Eq221}, \eqref{Eq37}).
Then, the  morphism  $\chi$ 
is unramified at $q$ (i.e., there are no nontrivial first-order deformations of $\msE^\spadesuit_+$ preserving the dormancy condition and  inducing the trivial deformation of the dormant $\mfs \mfo_{2\ell -1}$-oper $\msE^\spadesuit_{+\Rightarrow \emptyset}$) if  and only if
the following equality  holds:
\begin{align}
\left(H^0 (X,  \Omega^{\otimes \ell} \cap \mr{Im}(\nabla)) = \right) H^0 (X, (\Omega \otimes \mcF_{\varTheta_0}^{2\ell -2}) \cap \mr{Im}(\nabla)) = 0.
\end{align}
\epr
\begin{proof}
To begin with, we introduce some notation.
Each morphism of abelian sheaves $\nabla' : \mcK^0 \rightarrow \mcK^1$ is  identified with a complex concentrated at degrees  $0$ and $1$, and  we denote this complex  by   $\mcK^\bullet [\nabla']$.
Moreover,
we denote by $\mbH^i (X, \mcK^\bullet [\nabla'])$ ($i \geq 0$) the $i$-th hypercohomology group of $\mcK^\bullet [\nabla']$.
Given  a vector bundle $\mcG$ on $X$, we shall write $\mcE nd (\mcG)$ (resp., $\mcE nd^0 (\mcG)$)  for the sheaf of $\mcO_X$-linear endomorphisms  (resp., $\mcO_X$-linear endomorphisms with vanishing trace) of $\mcG$.

Now, let us consider the proof.
Denote by
 $\mcE nd^\circledast (\mcF_{\varTheta, +})$ the subbundle of $\mcE nd (\mcF_{\varTheta, +})$ consisting of  endomorphisms $h : \mcF_{\varTheta, +} \rightarrow \mcF_{\varTheta, +}$ such that
$\mr{Im}(h)  \subseteq \mcF_{\varTheta}$ and  its restriction $h |_{\mcF_\varTheta}  \in \mcE nd(\mcF_\varTheta)$ lies in $\mcE nd^0 (\mcF_\varTheta)$.
The $k$-connection $(\nabla_+^\diamondsuit)^\vee \otimes \nabla_+^\diamondsuit$ on $\mcE nd (\mcF_{\varTheta, +})$ $\left(= \mcF_{\varTheta, +}^\vee \otimes \mcF_{\varTheta, +} \right)$ 
(resp., $(\nabla^\diamondsuit)^\vee \otimes \nabla^\diamondsuit$ on $\mcE nd(\mcF_\varTheta)$ $\left(= \mcF_\varTheta^\vee \otimes \mcF_\varTheta \right)$)
restricts to a $k$-connection $\nabla^\mr{ad}_\circledast$ on $\mcE nd^\circledast (\mcF_{\varTheta, +})$ (resp., $\nabla^{\mr{ad}}$ on $\mcE nd^0 (\mcF_\varTheta)$).
The morphism 
$\mcE nd^\circledast (\mcF_{\varTheta, +}) \rightarrow \mcE nd^0 (\mcF_\varTheta)$ given by 
$h \mapsto h |_{\mcF_\varTheta}$ 
preserves the connection, so it 
gives
  a surjection $\mcK^\bullet [\nabla_\circledast^\mr{ad}] \rightarrow \mcK^\bullet [\nabla^\mr{ad}]$.
On the other hand,
the composite injection $\mcF_{\varTheta_0} \hookrightarrow \mcE nd(\mcF_{\varTheta, +})$ induced, via \eqref{Eq193},  by 
 the natural injection $\mcF_\varTheta \hookrightarrow \mcF_{\varTheta, +}$ and surjection $\mcF_{\varTheta, +} \twoheadrightarrow \overline{\mcF}_{\varTheta}$ 
 factors through 
 the inclusion $\mcE nd^\circledast (\mcF_{\varTheta, +}) \hookrightarrow \mcE nd (\mcF_{\varTheta, +})$;
 the resulting morphism $\mcF_{\varTheta_0} \rightarrow \mcE nd^\circledast (\mcF_{\varTheta, +})$ preserve the connection, i.e.,
 specifies a morphism $\mcK^\bullet [\nabla] \rightarrow \mcK^\bullet [\nabla^\mr{ad}_\circledast]$.
 Since $\mr{Ker} (\nabla^\mr{ad}) \left(= \mbH^0 (X, \mcK^\bullet [\nabla^\mr{ad}]) \right) = 0$ (cf. ~\cite[Proposition 6.5]{Wak8}),
 the resulting short exact sequence $0 \rightarrow \mcK^\bullet [\nabla] \rightarrow \mcK^\bullet [\nabla^\mr{ad}_\circledast] \rightarrow \mcK^\bullet [\nabla^\mr{ad}] \rightarrow 0$ induces  a diagram 
 \begin{align} \label{Erh367}
\vcenter{\xymatrix@C=26pt@R=36pt{
 0 \ar[r] & \mbH^1 (X, \mcK^\bullet [\nabla])\ar[d]^-{\alpha (\nabla)} \ar[r] & \mbH^1 (X, \mcK^\bullet [\nabla^\mr{ad}_\circledast]) \ar[r] \ar[d]^-{\alpha (\nabla^\mr{ad}_\circledast)} & \mbH^1 (X, \mcK^\bullet [\nabla^\mr{ad}]) \ar[d]^-{\alpha (\nabla)}  \\
 0 \ar[r] & H^0 (X, \mr{Coker}(\nabla)) \ar[r] & H^0 (X, \mr{Coker}(\nabla^\mr{ad}_\circledast)) \ar[r] & H^0 (X, \mr{Coker}(\nabla^\mr{ad})),
 }}
\end{align}
where both the upper and lower horizontal sequences are exact (cf.  ~\cite[Corollary 3.2.2]{Ogu2} and the  discussion following ~\cite[Proposition 1.2.4]{Ogu1} for the exactness of the lower sequence), and all the  vertical arrows are surjective because of 
 the conjugate spectral sequences associated to $\mcK^\bullet [\nabla]$, $\mcK^\bullet [\nabla^\mr{ad}_\circledast]$, and $\mcK^\bullet [\nabla^\mr{ad}]$ (cf. ~\cite[Eq.\,(758)]{Wak8}).
  Using the second arrow in the upper horizontal sequence,
 we obtain a composite injection
 \begin{align} \label{Eq201}
 H^0 (X, \Omega^{\otimes \ell}) \xrightarrow{\eqref{Eq177}} \mr{Coker}(H^0 (\nabla)) \hookrightarrow \mbH^1 (X, \mcK^\bullet [\nabla]) \hookrightarrow \mbH^1 (X, \mcK^\bullet [\nabla_\circledast^\mr{ad}]),
 \end{align}
 where the second arrow of this sequence denotes the injection arising from the Hodge to de Rham spectral sequence of $\mcK^\bullet [\nabla]$ (cf. ~\cite[\S\,6.2.1]{Wak8}).
 This composite injection allows us to consider   $H^0 (X, \Omega^{\otimes \ell})$ as a subspace of  $\mbH^1 (X, \mcK^\bullet [\nabla_\circledast^\mr{ad}])$.
Here,  let us consider 
the space of first-order deformations of  the flat vector bundle $(\mcF_{\varTheta, +}, \nabla_+^\diamondsuit)$ preserving  the filtration $\mcF_\varTheta \subseteq \mcF_{\varTheta, +}$ that induces
the trivial deformation of $\mr{det}(\mcF_\varTheta)$ 
 via restriction and induces the trivial deformation of $\overline{\mcF}_\varTheta \left(= \mcF_{\varTheta, +}/\mcF_\varTheta \right)$ via taking quotient.
According to  well-known generalities on the deformation theory of flat vector  bundles (cf. ~\cite[Chap.\,6]{Wak8}),
this space  is in bijection  with
(the underlying set of)  $\mbH^1 (X, \mcK^\bullet [\nabla_\circledast^\mr{ad}])$.
By this bijection,
the subspace $H^0 (X, \Omega^{\otimes \ell})$ of $\mbH^1 (X, \mcK^\bullet [\nabla_\circledast^\mr{ad}])$ may be identified with the deformation space $\mr{Def}(\nabla_{\sphericalangle, +}^\diamondsuit)$ of the $(\mr{GO}_{2\ell}^0, \vartheta)$-oper  $\nabla_{\sphericalangle, +}^\diamondsuit$  inducing the trivial deformation  of  $\nabla_{\sphericalangle, + \Rightarrow \emptyset}^\diamondsuit$ (cf. Proposition \ref{Prop155} and \eqref{Eq72}).
 Moreover, it follows from ~\cite[Proposition 6.11 and the discussion in \S\,6.4.1]{Wak8} that the subspace ${^p}\mr{Def}(\nabla_{\sphericalangle, +}^\diamondsuit)$ of $\mr{Def}(\nabla_{\sphericalangle, +}^\diamondsuit)$ consisting of deformations preserving the dormancy condition corresponds bijectively to
 $ H^0 (X, \Omega^{\otimes \ell}) \cap \mr{Ker} (\alpha (\nabla^\mr{ad}_\circledast))$, which coincides with $H^0 (X, \Omega^{\otimes \ell} \cap \mr{Im}(\nabla))$ by Lemma \ref{Lem778} proved below.
  Thus, 
 if $T_q \mcO p_{2\ell, g,r}^{^\mr{Zzz...}}$ (resp.,  $T_{\chi (q)} \mcO p_{2\ell -1, g,r}^{^\mr{Zzz...}}$) denotes the tangent space of  $ \mcO p_{2\ell, g,r}^{^\mr{Zzz...}}$ (resp., $\mcO p_{2\ell -1, g,r}^{^\mr{Zzz...}}$) at $q$ (resp., $\chi (q)$), then 
 the kernel  of the differential 
 \begin{align}
 d \chi : T_q \mcO p_{2\ell, g,r}^{^\mr{Zzz...}} \rightarrow T_{\chi (q)} \mcO p_{2\ell -1, g,r}^{^\mr{Zzz...}}
 \end{align}
  of $\chi$ at $q$  is isomorphic to  $H^0 (X, \Omega^{\otimes \ell} \cap \mr{Im}(\nabla))$.
 This implies the required equivalence.
\end{proof}
\SSP

The following lemma was applied in the proof of  the above proposition.

\SSP
\ble \label{Lem778}
Keeping the notation in the proof of Proposition \ref{Prop1001},
we obtain the equality
\begin{align}
H^0 (X, \Omega^{\otimes \ell}) \cap \mr{Ker}(\alpha (\nabla))  =
H^0 (X, \Omega^{\otimes \ell} \cap \mr{Im}(\nabla))
\end{align}
of subspaces of $H^0 (X, \Omega^{\otimes \ell})$.
\ele
\begin{proof}
Given an integer $n$ and a sheaf $\mcG$, we define the complex $\mcG [n]$ to be $\mcG$ (considered as a complex concentrated at degree $0$) shifted down by $n$, so that $\mcG [n]^{-n} = \mcG$ and $\mcG [n]^i = 0$ ($i \neq n$). 

Now, denote by $\nabla_\mr{Im}$ the morphism $\mcF_{\varTheta_0} \rightarrow \mr{Im}(\nabla)$ obtained from $\nabla$ by restricting its codomain to $\mr{Im}(\nabla)$.
The natural short exact sequence 
$0 \rightarrow \mcK^\bullet [\nabla_\mr{Im}] \rightarrow \mcK^\bullet [\nabla] \rightarrow \mr{Coker}(\nabla)[-1] \rightarrow 0$
 induces an exact  sequence of $k$-vector spaces
 \begin{align}
 0 \rightarrow \mbH^1 (X, \mcK^\bullet [\nabla_\mr{Im}]) \rightarrow \mbH^1 (X, \mcK^\bullet [\nabla]) \xrightarrow{\alpha (\nabla)} H^0 (X, \mr{Coker}(\nabla)).
 \end{align}
 By using the second arrow, we consider $\mbH^1 (X, \mcK^\bullet [\nabla_\mr{Im}])$ as a subspace of $\mbH^1 (X, \mcK^\bullet [\nabla])$.
In particular,  we have  $\mr{Ker}(\alpha (\nabla)) = \mbH^1 (X, \mcK^\bullet [\nabla_\mr{Im}])$,  which implies 
 \begin{align} \label{Eq803}
 H^0 (X, \Omega^{\otimes \ell}) \cap \mr{Ker}(\alpha (\nabla)) = H^0 (X, \Omega^{\otimes \ell}) \cap \mbH^1 (X, \mcK^\bullet [\nabla_\mr{Im}]).
 \end{align}
 
Next, 
the natural inclusion of short exact sequences 
\begin{align} \label{Erh469}
\vcenter{\xymatrix@C=26pt@R=36pt{
0 \ar[r] & \mr{Im}(\nabla)[-1] \ar[r] \ar[d]^-{\mr{inclusion}} & \mcK^\bullet [\nabla_\mr{Im}] \ar[r] \ar[d]^-{\mr{inclusion}} & \mcF_{\varTheta_0}[0] \ar[r] \ar[d]^-{\mr{id}}_-{\wr} & 0\\
0 \ar[r] & \Omega \otimes \mcF_{\varTheta_0} [-1]\ar[r] & \mcK^\bullet [\nabla] \ar[r] & \mcF_{\varTheta_0}[0] \ar[r] & 0
 }}
\end{align}
induces 
  a morphism of exact sequences
 \begin{align} \label{Erh467}
\vcenter{\xymatrix@C=26pt@R=36pt{
 H^0  (X, \mr{Im}(\nabla)) \ar[r]^-{e_{\mr{Im}, \sharp}} \ar[d] & \mbH^1 (X, \mcK^\bullet [\nabla_\mr{Im}]) \ar[r]^-{e_{\mr{Im}, \flat}} \ar[d]^-{\mr{inclusion}} & H^1 (X,  \mcF_{\varTheta_0}) \ar[d]_-{\wr}^-{\mr{id}} 
 \\
 H^0 (X, \Omega \otimes \mcF_{\varTheta_0}) \ar[r]_-{e_\sharp} & \mbH^1 (X, \mcK^\bullet [\nabla]) \ar[r]_-{e_\flat} &  H^1 (X, \mcF_{\varTheta_0}).
 }}
\end{align}
Since the injection $H^0 (X, \Omega^{\otimes \ell}) \hookrightarrow \mbH^1 (X, \mcK^\bullet [\nabla])$ (cf. \eqref{Eq201}) factors through  $e_\sharp$ (i.e., its image is contained in $\mr{Ker}(e_\flat)$), 
the diagram \eqref{Erh467} shows
\begin{align} \label{Eq800}
H^0 (X, \Omega^{\otimes \ell}) \cap \mbH^1 (X, \mcK^\bullet [\nabla_\mr{Im}]) = H^0 (X, \Omega^{\otimes \ell}) \cap \mr{Im}(e_{\mr{Im}, \sharp}) = H^0 (X, \Omega^{\otimes \ell} \cap \mr{Im}(\nabla)).
\end{align}
 Thus, the assertion follows from \eqref{Eq803} and \eqref{Eq800}.
\end{proof}
\SSP

By the above proposition,  the unramifiedness of $\chi$ amounts  to vanishing  the $k$-vector space $H^0 (X, \Omega^{\otimes \ell} \cap \mr{Im}(\nabla))$.
In what follows, we try to obtain a much better  understanding  of  that space.   
Recall from ~\cite[Theorem D]{Wak8} that
 the dual  $\nabla^\vee$
   of  $\nabla$ (in the sense of ~\cite[Eq.\,(705)]{Wak8})
     is isomorphic to
   the  $(\mr{GL}_{2\ell -1}, \vartheta_0)$-oper $D^{\clubsuit \Rightarrow \diamondsuit}$ arising from a $(2\ell -1, \vartheta_0)$-projective connection $D^\clubsuit$  on $\msX$ (cf. ~\cite[Definition 4.37 and  Eq.\,(529)]{Wak8}).
   (Note that the dual $\vartheta^\blacktriangledown$ of $\vartheta$, in the sense of ~\cite[Eq.\,(701)]{Wak8}, is isomorphic to $\vartheta$ itself.)
  Hence,  after applying  a suitable  gauge transformation,  we may  suppose  that $\nabla^\vee = D^{\clubsuit \Rightarrow \diamondsuit}$ for such a $D^{\clubsuit}$.
  We here abuse notation by writing $D^\clubsuit$ for  the corresponding differential operator $\Omega^{\otimes (-\ell+1)} \rightarrow \Omega^{\otimes \ell}$ via an isomorphism displayed in  ~\cite[Eq.\,(471)]{Wak8}.

\SSP
\ble \label{Lem99}
\begin{itemize}
\item[(i)]
The composite $\mr{Ker} (\nabla) \hookrightarrow \mcF_{\varTheta_0} \twoheadrightarrow \left(\mcF_{\varTheta_0}/ \mcF_{\varTheta_0}^0 = \right) \Omega^{\otimes (-\ell +1)}$  is injective.
Moreover, (when we consider $\mr{Ker}(\nabla)$ as a subsheaf of $\Omega^{\otimes (-\ell +1)}$ by using this injection) we have the equality $\mr{Ker} (\nabla) = \mr{Ker}(D^\clubsuit)$ of subsheaves of $\Omega^{\otimes (-\ell +1)}$.
\item[(ii)]
The equality $\Omega^{\otimes \ell} \cap\mr{Im}(\nabla)
= \mr{Im}(D^\clubsuit)$ between subsheaves of $\Omega^{\otimes \ell} \left(= \Omega \otimes \mcF_{\varTheta_0}^{2\ell -2} \right)$ holds. 
\end{itemize}
\ele
\begin{proof}
First, we shall consider the first assertion of  (i).
Suppose that there exists a nonzero local section $v$ of $\mr{Ker}(\nabla) \cap \mcF_{\varTheta_0}^1$.
Then, we can define  $j_0 := \mr{max} \left\{ j \, | \, v \in \mcF_{\varTheta_0}^j\right\} \left( \geq 1 \right)$.
The image $\overline{v}$ of $v$ via $\mcF_{\varTheta_0}^{j_0} \twoheadrightarrow \mcF_{\varTheta_0}^{j_0}/\mcF_{\varTheta_0}^{j_0+1}$ is nonzero.
Since the morphism $\mr{KS}^{j_0} : \mcF_{\varTheta_0}^{j_0}/\mcF_{\varTheta_0}^{j_0+1} \rightarrow \Omega \otimes (\mcF_{\varTheta_0}^{j_0-1}/\mcF_{\varTheta_0}^{j_0})$ (cf. \eqref{Eq356}) is an isomorphism,
the image 
 $\mr{KS}^{j_0} (\overline{v})$ is nonzero.
 But, by the definition of $\mr{KS}^{j_0}$, it contradicts the assumption that $v \in \mr{Ker}(\nabla)$.
 It follows that $\mr{Ker}(\nabla) \cap \mcF_{\varTheta_0}^1 = 0$, which implies the injectivity of the morphism $\mr{Ker}(\nabla) \rightarrow \Omega^{\otimes (-\ell +1)}$.
 
Next, to prove the second assertion of (i), let us observe that $\nabla$ can be locally described (with respective to a suitable local basis of $\mcF_{\varTheta_0}$ compatible, in a certain sense,  with $\{ \mcF^j_{\varTheta_0} \}_j$) as 
 \begin{align} \label{Eq301}
 \nabla = \partial +  \begin{pmatrix} 0 & q_2 & q_3 &  \cdots & q_{2\ell -2} & q_{2\ell -1} \\ -1 & 0 & 0 & \cdots & 0 & 0 \\
 0 & -1 & 0 & \cdots & 0 & 0 \\
 0 & 0 & -1 & \cdots & 0 & 0 \\
 \vdots & \vdots & \vdots & \cdots & \vdots & \vdots \\
 0 & 0 & 0 & \cdots & -1 & 0
   \end{pmatrix}
 \end{align}
for a local generator $\partial \in \mcT$ (viewed  as a locally defined derivation on $\mcO_X$) and some local functions $q_2, q_3, \cdots, q_{2\ell -1}$.
Then, the assignment 
 $y \cdot (\partial^\vee)^{-\ell +1}\mapsto {^t} (\partial^{2\ell -2}(y), \partial^{2\ell -3}(y), \cdots, \partial (y), y)$ gives a bijective correspondence between the local sections of $\mr{Ker}(D^\clubsuit)$ and the local sections of $\mr{Ker}(\nabla)$.
This completes the proof of assertion (i).

Finally, the desired equality in assertion (ii) is given, with respect to the local description \eqref{Eq301}, by the well-defined correspondence $D^\clubsuit (y\cdot (\partial^\vee)^{-\ell +1}) \mapsto \nabla ({^t} (\partial^{2\ell -2}(y), \partial^{2\ell -3}(y), \cdots, \partial (y), y))$ for each local function $y$.
\end{proof}
\SSP

\bpr \label{Prop100}
Let us consider $\mr{Ker}(\nabla)$ and $\Omega^{\otimes \ell} \cap \mr{Im}(\nabla)$ as 
vector bundles on $X^{(1)}$
via the underlying homeomorphism of $F$.
Then, there exists a canonical short exact  sequence 
\begin{align} \label{Eq200}
0 \rightarrow \mr{Ker} (\nabla) \xrightarrow{} F_{*} (\Omega^{\otimes (-\ell +1)}) \xrightarrow{} 
\Omega^{\otimes \ell} \cap \mr{Im}(\nabla)
\rightarrow 0
\end{align}
of 
vector bundles on $X^{(1)}$.
In particular, the vector bundle $\Omega^{\otimes \ell} \cap \mr{Im}(\nabla)$ has rank $p-2\ell +1$.
 \epr
\begin{proof}
The desired sequence can be obtained, via $F_*(-)$, from the natural short exact sequence
\begin{align}
0 \rightarrow \mr{Ker}(D^\clubsuit) \xrightarrow{\mr{inclusion}} \Omega^{\otimes (-\ell +1)}
\xrightarrow{D^\clubsuit} \mr{Im}(D^\clubsuit) \rightarrow 0
\end{align}
under the identifications $\mr{Ker}(\nabla) = \mr{Ker}(D^\clubsuit)$ and $\Omega^{\otimes \ell} \cap \mr{Im}(\nabla) = \mr{Im}(D^\clubsuit)$ resulting from Lemma \ref{Lem99}, (i) and (ii), respectively.
\end{proof}

\LSP
\subsection{The case of  a $3$-pointed projective line} \label{SS3r2}

Denote by $\mbP$ the projective line over $k$ and by $[0]$, $[1]$, $[\infty]$ the $k$-rational points of $\mbP$ determined by the values $0$, $1$, and $\infty$, respectively.
After ordering the points $[0]$, $[1]$, $[\infty]$, we obtain a unique (up to isomorphism) $3$-pointed stable curve
\begin{align}
\msP := (\mbP/k, \{ [0], [1], [\infty] \})
\end{align} 
of genus $0$ over $k$.
We shall keep the notation in the previous subsection and moreover suppose that $(g, r) = (0, 3)$ and $\msX = \msP$.

\SSP
\ble \label{Lem56}
The vector bundle $\mr{Ker}(\nabla)$ on the Frobenius twist $\mbP^{(1)}$ of $\mbP$ has degree $-3 (\ell -1)$.
\ele
\begin{proof}
For $q \in \{ 0, 1, \infty \}$,
let  $(- a_{q, 1}, \cdots, - a_{q, 2\ell -1})$  (where $0\leq a_{q, 1} \leq \cdots \leq a_{q, 2\ell -1} < p$) be  the exponent of $\nabla$ at the point $[
q]$, in the sense of ~\cite[Definition 8.2]{Wak8}.
According to ~\cite[Proposition 8.4, (ii)]{Wak8},
the integers $a_{q, 1}, \cdots, a_{q, 2\ell -1}$
 are mutually distinct.
 Hence, we may assume that $a_{q,1} < a_{q, 2} < \cdots < a_{q, 2\ell -1}$.
 The exponent of the dual $(\mcF_{\varTheta_0}^\vee, \nabla^\vee)$ at $[q]$
 coincides with $(0, -(p-a_{q, 2\ell -1}), - (p-a_{q, 2\ell -2}), \cdots, - (p-a_{q, 2}))$ (resp., $(-(p-a_{q, 2\ell -1}), - (p-a_{q, 2\ell -2}), \cdots, - (p-a_{q, 1}))$) if $a_{q, 1} = 0$ (resp., $a_{q, 1} \neq 0$).
 But, since $(\mcF_{\varTheta_0}, \nabla)$ is isomorphic to its dual because of the nondegeneracy of the bilinear map $\omega_0 : \mcF_{\varTheta_0}^{\otimes 2} \rightarrow \mcO_X$,
the resp'd situation cannot occur.
It follows that   $a_{q, 1} = 0$ and $a_{q, m} = p-a_{q, 2\ell +1-m}$ for every $m =2, \cdots, 2\ell -1$.
 Hence, we have 
 \begin{align}
 \mr{deg}(\mr{Ker}(\nabla)) & = \frac{1}{p} \cdot \mr{deg} (F^*(\mr{Ker}(\nabla))) \\
 & = \frac{1}{p} \cdot \left(\mr{deg}(\mcF_{\varTheta_0}) - \sum_{q \in \{ 0, 1, \infty \}} \sum_{m=1}^{2\ell -1} a_{q, m}\right) \notag \\
 & = \frac{1}{p} \cdot \left(\sum_{j=0}^{2\ell -2} \mr{det}(\mcF_{\varTheta_0}^j/ \mcF_{\varTheta_0}^{j+1}) - \sum_{q \in \{ 0, 1, \infty \}} \sum_{m=2}^{\ell } (a_{q, m}  + a_{q, 2\ell +1-m}) \right) \notag \\
 & = \frac{1}{p} \cdot \left(\sum_{j=0}^{2\ell -2}  \mr{det}(\Omega^{\otimes (-\ell +1-j)}) - \sum_{q \in \{ 0, 1, \infty \}} \sum_{m=2}^{\ell } p \right) \notag \\
& = \frac{1}{p} \left( 0 - 3 p  (\ell -1)\right) \notag \\
 & = - 3  (\ell -1),
 \end{align}
 where the second equality follows from ~\cite[Lemma 8.1, (ii)]{Wak8}.
This completes the proof of this lemma.
\end{proof}
\SSP

Next, recall the Birkhoff-Grothendieck theorem, asserting that any vector bundle on the projective line    is isomorphic to a direct sum of line bundles.
Hence, there exists a sequence of  integers $w_1, \cdots, w_{p-2\ell +1}$ with $w_1 \leq \cdots \leq w_{p-2\ell +1}$ such that
\begin{align} \label{Eq270}
\Omega^{\otimes \ell} \cap \mr{Im} (\nabla) \cong \bigoplus_{j=1}^{p-2\ell +1} \mcO_{\mbP^{(1)}} (w_j)
\end{align}
(cf. the second assertion of Proposition \ref{Prop100}).

\SSP
\ble \label{Lem123}
Let us keep the above notation.
Then,
we have $w_1 \geq -2$.
\ele
\begin{proof}
For simplicity, we write 
 $\mcA := F^*(F_*(\Omega^{\otimes (-\ell +1)}))$
 and $\mcB := F^* (\Omega^{\otimes \ell} \cap \mr{Im}(\nabla))$, which are  vector bundles on $\mbP$ of rank $p$ and $p-2\ell +1$, respectively.
The pull-back of \eqref{Eq200} define a short exact sequence of $\mcO_{\mbP}$-module
\begin{align}
0 \rightarrow F^*(\mr{Ker}(\nabla)) \xrightarrow{\gamma_\sharp} \mcA \xrightarrow{\gamma_\flat} \mcB
\rightarrow 0.
\end{align}
Let 
 $\{ \mcA^j \}_{j=0}^p$  be the $p$-step decreasing filtration on $\mcA$ constructed as in ~\cite[Eq.\,(1143)]{Wak8}.
 To be precise,  it is defined as follows:
 \begin{align}
 \mcA^0 &:= \mcA; \\
 \mcA^1 & := \mr{Ker} \left(\mcA \xrightarrow{\xi} \Omega^{\otimes (-\ell +1)} \right);  \notag \\
 \mcA^j & := \mr{Ker} \left( \mcA^{j-1} \xrightarrow{\nabla^\mr{can}|_{\mcA^{j-1}}} \Omega \otimes \mcA \xrightarrow{\mr{quotient}} \Omega \otimes (\mcA/\mcA^{j-1})\right) \  (j=2, \cdots, p),
 \end{align}
 where $\xi$ denotes the morphism corresponding to the identity morphism of $F_*(\Omega^{\otimes (-\ell +1)})$ via the adjunction relation ``$F^*(-) \dashv F_*(-)$", and  $\nabla^\mr{can}$ denotes the canonical connection on $\mcA$ determined uniquely by the condition that the local sections in $F^{-1}(F_{*}(\Omega^{\otimes (-\ell +1)}))$  are horizontal (cf. ~\cite[Eq.\,(630)]{Wak8}).
 This 
  gives a filtration $\{ \mcB^j \}_{j=2\ell -1}^{p}$ on $\mcB$ in such a way that $\mcB^j := \mcB$ if $j=2\ell -1$ and  
 $\mcB^j := \gamma_\flat (\mcA^j)$  if $j=2\ell, \cdots, p$.
Similarly to the first part in the proof of ~\cite[Proposition 9.2]{Wak8},
it is verified  that the composite 
\begin{align} \label{Eq289}
F^*(\mr{Ker}(\nabla)) \xrightarrow{\gamma_\sharp} \mcA \twoheadrightarrow \mcA /\mcA^{2\ell -1},
\end{align}
is injective.
It follows that  the composite
\begin{align} \label{Eq399}
\mcA^{2\ell -1} \xrightarrow{\mr{inclusion}} \mcA \xrightarrow{\gamma_\flat} \mcB
\end{align}
is injective and moreover bijective  over the generic point.
This composite induces a nonzero injection $\mcA^j/\mcA^{j+1} \hookrightarrow \mcB^j/\mcB^{j+1}$ ($j=2\ell -1, \cdots, p-1$).
Hence, 
 for each $j=2\ell, \cdots, p-1$, we have 
 \begin{align} \label{Eq290}
\mr{det}(\mcB^j/\mcB^{j+1}) \geq \mr{deg}(\mcA^j /\mcA^{j+1}) =  \mr{deg}(\Omega^{\otimes (-\ell +1)} \otimes \Omega_{\mbP/k}^{\otimes j}) = -\ell +1 - 2j,
 \end{align}
  where $\Omega_{\mbP/k}$ denotes the sheaf of non-logarithmic $1$-forms on $\mbP$ over $k$, and  the first ``$=$" follows from 
  ~\cite[Proposition 9.1]{Wak8}. 
  Note that  the ``$\geq$" in this sequence becomes an equality ``$=$" when $j > 2\ell -1$. 
  In particular, the following equality holds:
  \begin{align} \label{Eq276}
  \mr{deg}(\mcB^{p-1}) = -\ell -2p+3.
  \end{align}

Next, we shall write
$\xi : \mcB \twoheadrightarrow \mcO_{\mbP^{(1)}} (p \cdot w_1)$ 
for
the projection onto the $1$-st factor with respect to 
the 
decomposition  $\mcB \cong \bigoplus_{j=1}^{p-2\ell +1} \mcO_{\mbP}(p\cdot w_j)$ obtained as the pull-back of  \eqref{Eq270}.
Also,
write
\begin{align}
j_0 := \mr{max}\left\{ j \, | \, 2\ell -1 \leq j \leq p-1, \ \xi (\mcB^j) \neq 0\right\}.
\end{align}
Then, 
$\xi$ induces a {\it nonzero} morphism between line bundles
$\overline{\xi} : \mcB^{j_0}/\mcB^{j_0 +1} \rightarrow \mcO_{\mbP^{(1)}} (p \cdot w_1)$.
In particular, 
$\overline{\xi}$ is injective, so we have
\begin{align} \label{Eq285}
p \cdot w_1 &= \mr{deg}(\mcO_{\mbP^{(1)}}(p\cdot w_1)) \geq \mr{deg}(\mcB^{j_0}/\mcB^{j_0+1}) \geq \mr{deg}(\mcB^{p-1}),
\end{align}
where the last inequality
 follows from
the sequence
 \begin{align} \label{Eq280}
\mr{deg}(\mcB^{p-1}) < \mr{deg} (\mcB^{p-2}/\mcB^{p-1}) < \mr{deg}(\mcB^{p-3}/\mcB^{p-2}) < \cdots < \mr{deg} (\mcB/\mcB^{2\ell})
\end{align}
induced from \eqref{Eq290}.
By combining  \eqref{Eq276}
and \eqref{Eq285}, we obtain 
$w_1 \geq  \frac{-\ell- 2 p+3}{p}$ $\left(= -2 - \frac{\ell -3}{p} \right)$.
Thus, 
the desired inequality follows from this inequality together with the assumption that $w_1 \in \mbZ$  and $p > 2 (2\ell -1)$.
\end{proof}
\SSP

By applying the above lemma, we can prove the following assertion.

\SSP
\bpr \label{Prop30}
Under the assumption that $\msX = \msP$,
let $\nabla$ be as introduced at the beginning of \S\,\ref{SS32}.
Suppose that $\ell > 3$ and that the dual $(\mr{GL}_{2\ell -1}, \vartheta_0)$-oper   $\nabla^\vee$ of $\nabla$ coincides with $D^{\clubsuit \Rightarrow \diamondsuit}$ for some $(2\ell -1, \vartheta_0)$-projective connection $D^\clubsuit$ on $\msP$. 
Then, the sheaf $\Omega^{\otimes \ell} \cap \mr{Im}(\nabla)$, considered as a vector bundle on $\mbP^{(1)}$, satisfies
\begin{align} \label{Eq456}
\Omega^{\otimes \ell} \cap \mr{Im}(\nabla) \cong \mcO_{\mbP^{(1)}} (-1)^{\oplus (p-2\ell +1)}.
\end{align}
In particular, 
the equality  $H^0 (\mbP^1, \Omega^{\otimes \ell} \cap \mr{Im}(\nabla)) = 0$ holds.
\epr
\begin{proof}
 It is well-known that, for each integer $m$,   the direct image  $F_{*}(\mcO_{\mbP} (m))$ of the line bundle $\mcO_{\mbP}(m)$ is isomorphic to $\mcO_{\mbP^{(1)}} (m) \oplus \mcO_{\mbP^{(1)}}(-1)^{\oplus (p-1)}$.
  In particular, 
  we have 
  \begin{align} \label{Eq457}
  F_{*}(\Omega^{\otimes (-\ell +1)}) \cong \mcO_{\mbP^{(1)}} (-\ell +1) \oplus \mcO_{\mbP^{(1)}} (-1)^{\oplus (p-1)}.
  \end{align}
Let us fix a decomposition \eqref{Eq270}, and
 observe the following sequence of equalities:
 \begin{align} \label{Eq489}
 \sum_{j=1}^{p-2\ell +1} w_j &= \mr{deg}(\Omega^{\otimes \ell} \cap \mr{Im}(\nabla))  \\
 &= \mr{deg}(F_*(\Omega^{\otimes (-\ell +1)})) - \mr{deg}(\mr{Ker}(\nabla)) \notag \\
 & = \left(\mr{deg}(\mcO_{\mbP^{(1)}}(-\ell +1)) + (p-1) \cdot \mr{deg}(\mcO_{\mbP^{(1)}}(-1)) \right) -\mr{deg}(\mr{Ker}(\nabla)) \notag \\
 & = (-\ell -p +2) - (-3 (\ell -1)) \notag \\
 & = 2\ell -1  -p,
 \end{align}
 where the second equality follows from Proposition \ref{Prop100},
  the third equality follows from
  \eqref{Eq457}, and the fourth equality follows from
   Lemma \ref{Lem56}.
On the other hand,
the surjection $F_*(\Omega^{\otimes (-\ell +1)}) \twoheadrightarrow \Omega^{\otimes \ell} \cap \mr{Im}(\nabla)$ in \eqref{Eq200} induces, via  \eqref{Eq456} and \eqref{Eq457},
a surjection $\mcO_{\mbP^{(1)}} (-\ell +1) \oplus \mcO_{\mbP^{(1)}} (-1)^{\oplus (p-1)}\twoheadrightarrow \bigoplus_{j=1}^{p-2\ell +1} \mcO_{\mbP^{(1)}} (w_j)$.
This implies that (since $-\ell +1 < -2 \leq w_1$  by Lemma \ref{Lem123} and our  assumption)
$w_1$ must be greater than $-2$.
 Hence, it follows from \eqref{Eq489} that $(w_1, \cdots, w_{p-2\ell +1}) = (-1, \cdots, -1)$.
 This completes the proof of this proposition.
 \end{proof}
\SSP

\bco \label{Cor34}
(Recall that $\overline{\mcM}_{0, 3}$ is isomorphic to $\mr{Spec}(k)$.)
Suppose that $\frac{p+2}{4} > \ell > 3$.
Then, 
the $k$-scheme  $\mcO p^{^\mr{Zzz...}}_{2\ell, 0, 3}$ is isomorphic to the disjoint union of finitely many copies of $\mr{Spec}(k)$.
\eco
\begin{proof}
By  ~\cite[\S\,4.6.4]{Wak8},
there exists a $(2\ell-1)$-theta characteristic  $\vartheta := (\varTheta, \nabla_\vartheta)$ such that $\nabla_\vartheta$ has vanishing $p$-curvature.
Hence, the dormant $\mfs \mfo_{2\ell}$-oper classified by a point of $\mcO p^{^\mr{Zzz...}}_{2\ell, 0, 3}$ arises from a dormant $(\mr{GO}_{2\ell}^0, \vartheta)$-oper (cf. Proposition  \ref{Prop15}).
By  Propositions \ref{Prop1001} and  \ref{Prop30},
the morphism 
$\chi : \mcO p^{^\mr{Zzz...}}_{\mfs \mfo_{2\ell}, 0, 3} \rightarrow \mcO p^{^\mr{Zzz...}}_{\mfs \mfo_{2\ell -1}, 0, 3}$ is verified to be  unramified.
On the other hand, since we have assumed the inequality $p > 2 (2\ell -1)$,  it follows from   ~\cite[Theorem G]{Wak8}  that  $\mcO p^{^\mr{Zzz...}}_{\mfs \mfo_{2\ell -1}, 0, 3}$ is \'{e}tale over $k$, i.e., isomorphic to the  disjoint union of finitely many copies of $\mr{Spec}(k)$.
This implies the \'{e}taleness of  $\mcO p^{^\mr{Zzz...}}_{\mfs \mfo_{2\ell}, 0, 3}$ over $k$, which 
 completes the  proof of this assertion.
\end{proof}

\LSP
\subsection{The generic \'{e}taleness of the moduli space} \label{SS3rr2}

Applying  Corollary \ref{Cor34} and a result in ~\cite{Wak8},
we obtain the following Theorem \ref{Thm1}.
To describe it, recall that 
a pointed stable curve over $k$ is called {\bf totally degenerate}  if it is obtained by gluing together finitely many copies of $\msP$ along their marked points (cf. ~\cite[Definition 7.15]{Wak8} for its precise definition).

\SSP
\bt[cf. Theorem \ref{ThA}] \label{Thm1}
Suppose that $\frac{p+2}{4} > \ell > 3$.
Then,
the stack  $\mcO p_{2 \ell, g, r}^{^\mr{Zzz...}}$ 
is \'{e}tale over the points of 
$\overline{\mcM}_{g, r}$ classifying totally degenerate curves.
In particular, (because of the irreducibility of $\overline{\mcM}_{g, r}$ and the finiteness of $\mcO p_{2 \ell, g, r}^{^\mr{Zzz...}}/\overline{\mcM}_{g, r}$)  $\mcO p_{2 \ell, g, r}^{^\mr{Zzz...}}$ is generically \'{e}tale over $\overline{\mcM}_{g, r}$, i.e., any irreducible component that dominates $\overline{\mcM}_{g, r}$ admits a dense open substack which   is  \'{e}tale over $\overline{\mcM}_{g, r}$.
\et
\begin{proof}
The assertion follows from Corollary \ref{Cor34} and  ~\cite[Proposition 7.19]{Wak8}.
\end{proof}
\SSP

We conclude the present paper by describing  a factorization property of the generic degree $\mr{deg}(\Pi_{2\ell, g, r})$ of the morphism $\Pi_{2\ell, g, r}$
  in accordance with the data of radii (cf. ~\cite[Chap.\,7]{Wak8} for the previous study of related topics).

Let us suppose that $\frac{p+2}{4} > \ell > 3$. 
Denote by $\mfc$ the GIT quotient of  $\mfs \mfo_{2\ell}$ by the adjoint action of
$\mr{PGO}_{2\ell}^0$. 
Since $\mfc$ can be defined over $\mbF_p$, it makes sense to speak of the set of $\mbF_p$-rational points of $\mfc$, denoted by $\mfc (\mbF_p)$.

Given an $r$-tuple $\rho := (\rho_i)_{i=1}^r \in \mfc (\mbF_p)^{\times r} \left(= \mfc (\mbF_p) \times \cdots \times \mfc (\mbF_p)\right)$ and a dormant $\mfs \mfo_{2\ell}$-oper $\msE^\spadesuit_+ := (\mcE_B, \nabla)$ on an $r$-pointed stable curve $\msX$,
we say that  $\msE^\spadesuit_+$ is {\bf of radii $\rho$}
if, for every $i=1, \cdots, r$,  
the residue of $\nabla$ (as an element of  $\mfs \mfo_{2\ell}$)  at the $i$-th marked point of $\msX$  is mapped to $\rho_i$ via the quotient $\mfs \mfo_{2\ell} \twoheadrightarrow \mfc$  (cf. ~\cite[Definition 2.32]{Wak8}).

We denote by
\begin{align}
\mcO p^{^\mr{Zzz...}}_{2\ell, \rho, g, r}
\end{align}
the (possibly empty) closed and open substack of $\mcO p^{^\mr{Zzz...}}_{2\ell, g, r}$ classifying dormant $\mfs \mfo_{2\ell}$-opers {\it of radii $\rho$}, which admits the projection
\begin{align}
\Pi_{2\ell, \rho, g, r} : \mcO p^{^\mr{Zzz...}}_{2\ell, \rho, g, r} \rightarrow \overline{\mcM}_{g, r}.
\end{align}
By Theorem \ref{Thm1}, the stack $\mcO p^{^\mr{Zzz...}}_{2\ell, \rho, g, r}$ is (finite and) generically \'{e}tale, so it makes sense to speak of  the generic degree
$\mr{deg}(\Pi_{2\ell, \rho, g, r})$ of $\Pi_{2\ell, \rho, g, r}$.
Since  $\mcO p^{^\mr{Zzz...}}_{2\ell, g, r}$ decomposes into the direct sum
$\mcO p^{^\mr{Zzz...}}_{2\ell, g, r} = \coprod_{\rho \in \mfc (\mbF_p)^{\times r}}\mcO p^{^\mr{Zzz...}}_{2\ell, \rho, g, r}$ (cf. ~\cite[Theorem C]{Wak8}),
the equality $\mr{deg}(\Pi_{2 \ell, g, r}) = \sum_{\rho \in \mfc (\mbF_p)^{\times r}} \mr{deg}(\Pi_{2\ell, \rho, g, r})$ holds.

If $\star : \mbG_m \times \mfc \rightarrow \mfc$ denotes the $\mbG_m$-action on $\mfc$ coming from homotheties on $\mfs \mfo_{2\ell}$ (cf. ~\cite[Eq.\,(264)]{Wak8}),
then we have $\lambda = (-1)\star \lambda$  for any $\lambda \in \mfc (\mbF_p)$. 
This fact together with   Corollary \ref{Cor34} shows that $\mfs \mfo_{2\ell}$ satisfies 
both the conditions $(*)$ and $(**)$ described at the beginning of ~\cite[\S\,7.3.5]{Wak8}.
Thus, according to ~\cite[Proposition 7.33]{Wak8} and the discussion in 
~\cite[\S\,7.4]{Wak8},
one can obtain  the {\it pseudo-fusion ring}  for dormant $\mfs \mfo_{2\ell}$-opers $\Fus$, in the sense of ~\cite[Definition 7.34]{Wak8}.
To be precise, $\Fus$ is defined as the unitization of the free abelian group $\mbZ^{\mfc (\mbF_p)}$ with basis $\mfc (\mbF_p)$ equipped with the multiplication $\ast : \mbZ^{\mfc (\mbF_p)} \times \mbZ^{\mfc (\mbF_p)} \rightarrow \mbZ^{\mfc (\mbF_p)}$ given by
\begin{align}
\alpha \ast \beta = \sum_{\lambda \in \mfc (\mbF_p)} \mr{deg}(\Pi_{2\ell, (\alpha, \beta, \lambda), 0, 3}) \cdot \lambda.
\end{align}
The explicit understanding of its ring structure allows us to perform a computation of
 the values $\mr{deg}(\Pi_{2\ell, \rho, g, r})$.
In fact, we obtain  the following  assertion.

\SSP
\bt[cf.  Theorem \ref{ThC}] \label{Thm15}
Write $\mfS$ for the set of ring homomorphims $\Fus \rightarrow \mbC$ and write $\mr{Cas} := \sum_{\lambda \in \mfc (\mbF_p)} \lambda \ast \lambda \left(\in \Fus \right)$.
Then, for each $\rho := (\rho_i)_{i=1}^r \in \mfc (\mbF_p)^{\times r}$,
the following equality holds:
\begin{align}
\mr{deg}(\Pi_{2\ell, \rho, g, r}) = \sum_{\chi \in \mfS} \chi (\mr{Cas})^{g-1} \cdot \prod_{i=1}^r \chi (\rho_i).
\end{align}
In particular, if $r = 0$ (which implies $g > 1$), then this equality reads
\begin{align}
\mr{deg}(\Pi_{2\ell, \emptyset, g, 0}) = \sum_{\chi \in \mfS} \chi (\mr{Cas})^{g-1}.
\end{align}
\et
\begin{proof}
The assertion follows from ~\cite[Theorem 7.36, (ii)]{Wak8}.
\end{proof}

\LSP
\subsection*{Acknowledgements} 
We are grateful for the many constructive conversations we had with  the {\it moduli space of dormant $\mfs \mfo_{2 \ell}$-opers}, who lives in the world of mathematics!
Our work was partially supported by Grant-in-Aid for Scientific Research (KAKENHI No. 21K13770).

\end{document}